\documentclass[11pt]{amsart}
\usepackage{amsmath}
\usepackage{amsfonts,amssymb,amsthm, txfonts, pxfonts,amscd}
\def\struckint{\mathop{%
\def\mathpalette##1##2{\mathchoice{##1\displaystyle##2}%
  {##1\textstyle##2}{##1\scriptstyle##2}{##1\scriptscriptstyle##2}}%
\mathpalette
{\vbox\bgroup\baselineskip0pt\lineskiplimit-1000pt\lineskip-1000pt
\halign\bgroup\hfill$}
{##$\hfill\cr{\intop}\cr\diagup\cr\egroup\egroup}%
}\limits}

\newtheorem{proposition}{Proposition}[section]
\newtheorem{theorem}[proposition]{Theorem}
\newtheorem{lemma}[proposition]{Lemma}
\newtheorem{corollary}[proposition]{Corollary}
\newtheorem{assumption}[proposition]{Assumption}

\theoremstyle{definition}
\newtheorem{definition}[proposition]{Definition}

\newtheorem{note}[proposition]{Note}

\newenvironment{pf*}[1]{\medskip \noindent {\em #1.} }{\endproof \medskip}

\newcommand{\xclass}[1]{\langle #1 \rangle}

\newcommand{\zz}[1]{\mathbb #1}

\title[Random Walk on a Surface Group]{Random Walk on a Surface Group:
 Behavior of the Green's Function at the Spectral
Radius}
\author{Steven P. Lalley} \address{University of Chicago\\ Department
of Statistics \\ 5734
University Avenue \\
Chicago IL 60637}
\email{lalley@galton.uchicago.edu}
\date{\today}
\subjclass{Primary 31C20, secondary 31C25 60J50 60B99}
\keywords{hyperbolic group, surface group, random walk, Green's function, Gromov
boundary, Martin boundary, Ruelle operator theorem, Gibbs state}
\thanks{Supported by NSF grant DMS  - 0805755}

\begin{document}

\begin{abstract}
It is proved that the Green's function of the simple random walk on a
surface group of large genus decays exponentially in distance at the
(inverse) spectral radius. It is also shown that Ancona's inequalities
extend to the spectral radius $R$, and therefore that the Martin
boundary for $R-$potentials coincides with the natural geometric
boundary $S^{1}$, and that the Martin kernel is uniformly H\"{o}lder
continuous. Finally, it is shown that the uniform H\"{o}lder
continuity of the Martin kernel up to the spectral radius implies that
the Green's function obeys a power law with exponent $1/2$.
\end{abstract}

\maketitle

\section{Introduction}\label{sec:introduction}

\subsection{Green's function and Martin boundary}\label{ssec:grfmb}
A (right) \emph{random walk} on a countable group $\Gamma$ is a discrete-time
Markov chain whose transition probabilities are $\Gamma -$invariant;
equivalently, it is a stochastic process $\{X_{n} \}_{n\geq 0}$ of the
form
\begin{equation}\label{eq:rw}
	X_{n}=x\xi_{1}\xi_{2}\dotsb \xi_{n}
\end{equation}
where $\xi_{1},\xi_{2},\dotsc $ are independent, identically
distributed $\Gamma -$valued random variables. The distribution of
$\xi_{i}$ is the \emph{step distribution} of the random walk.  The
\emph{Green's function} is the kernel of the resolvent operator
$r^{-1}(r^{-1}I-\zz{P})^{-1}$, where $\zz{P}$ is the transition probability
operator of the random walk. Equivalently, it is the generating function
of the transition probabilities: for $x,y\in \Gamma$ and $0\leq r < 1$
it is defined by the absolutely convergent series
\begin{equation}\label{eq:green}
	G_{r} (x,y):=\sum_{n=0}^{\infty} P^{x}\{X_{n}=y \}r^{n} =G_{r} (1,x^{-1}y);
\end{equation}
here $P^{x}$ is the probability measure on path space governing the
random walk with initial point $x$.  If the random walk is irreducible
(that is, if the support of the step distribution generates $\Gamma$)
then the radius of convergence $R$ of the series \eqref{eq:green} is
the same for all pairs $x,y$, and $1/R$ is the \emph{spectral radius}
of the transition operator. By a fundamental theorem of Kesten
\cite{kesten}, if the group $\Gamma$ is finitely generated and
nonamenable then $R>1$. Moreover, in this case the Green's function is
finite at its radius of convergence: for all $x,y\in\Gamma$,
\begin{equation}\label{eq:finiteAtR}
	G_{R} (x,y)<\infty.
\end{equation}

The Green's function is of central importance in the potential theory
associated with the random walk: in particular, it determines the
Martin boundary for $r-$potential theory. A prominent theme in the
study of random walks on nonabelian groups has been the relationship
between the geometry of the group and the nature of the Martin
boundary. A landmark result here is a theorem of Ancona~\cite{ancona}
describing the Martin boundary for random walks with finitely
supported step distributions on \emph{hyperbolic} groups: Ancona
proves that for every $r\in (0,R)$ the Martin boundary for
$r-$potential theory coincides with the \emph{geometric} (Gromov)
boundary, in a sense made precise below.  (Series~\cite{series} had
earlier established this in the special case $r=1$ when the group is
co-cocompact Fuchsian. See also \cite{anderson-schoen} and
\cite{ancona:annals} for related results concerning Laplace-Beltrami
operators on Cartan manifolds.)

It is natural to ask whether Ancona's theorem extends to $r=R$, that
is, if the Martin boundary is stable (see \cite{picardello-woess} for
the terminology) through the entire range $(0,R]$. One of the main
results of this paper (Theorem~\ref{theorem:1}) provides an
affirmative answer in the special case of simple random walk on the
\emph{surface groups} $\Gamma_{g}$. Let $A=A_{g}$ be the standard
symmetric set of generators for $\Gamma_{g}$:
\begin{equation}\label{eq:surfaceGenerators}
	A_{g}=\{a_{i}^{\pm 1},b_{i}^{\pm 1} \}_{1\leq i\leq g};
\end{equation}
these generators satisfy the fundamental relation
\begin{equation}\label{eq:surfaceRelations}
	\prod_{i=1}^{g} a_{i}b_{i}a_{i}^{-1}b_{i}^{-1}=1.
\end{equation}
The \emph{Cayley graph} $G^{\Gamma}$ of $\Gamma$ relative to the
generating set $A_{g}$ is the graph whose vertices are the elements of
$\Gamma $, and whose edges are the (unordered) pairs $x,y\in \Gamma$
such that $y=xa$ for some $a\in A_{g}$.  By \emph{simple} random walk
on $\Gamma_{g}$ (or $G^{\Gamma }$) we mean the random walk on $\Gamma$
whose step distribution is the uniform probability distribution on the
set $A_{g}$ of standard generators.  The surface group $\Gamma_{g}$
acts as a co-compact discrete group of isometries of the hyperbolic
plane, and so its Cayley graph can be embedded quasi-isometrically in
the hyperbolic plane; this implies that the Gromov boundary is the
circle $S^{1}$ at infinity.  Our main result (Theorem~\ref{theorem:1}
below) will directly imply the following.

\begin{theorem}\label{theorem:martinBoundary}
For simple random walk on a surface group $\Gamma_{g}$ of
sufficiently large genus $g$, the Martin boundary for $R-$potentials
coincides with the geometric boundary $S^{1}=\partial \Gamma_{g}$.
\end{theorem}

This assertion means that (1) for every geodesic ray
$y_{0},y_{1},y_{2},\dotsc$ in the Cayley graph that converges to a
point$\zeta \in \partial \Gamma$  and for every $x\in \Gamma$,
\begin{equation}\label{eq:martinConvergence}
	\lim_{n \rightarrow \infty} \frac{G_{R} (x,y_{n})}{G_{R}
	(1,y_{n})}=K_{R}(x,\zeta )=K (x,\zeta)
\end{equation}
exists; (2) for each $\zeta \in \partial \Gamma$ the function
$K_{\zeta} (x):= K(x,\zeta )$  is minimal positive $R-$harmonic in $x$;  (3) for
distinct points $\zeta ,\zeta '\in \partial \Gamma$ the functions
$K_{\zeta}$ and $K_{\zeta '}$ are different; and (4) the topology of
pointwise convergence on $\{K_{\zeta} \}_{\zeta \in \partial \Gamma}$
coincides with the usual topology on $\partial \Gamma=S^{1}$.

Our  results also yield explicit rates for the convergence
\eqref{eq:martinConvergence}, and imply that the Martin kernel $K_{r}
(x,\zeta)$ is \emph{H\"{o}lder} continuous in $\zeta$ relative to the
usual Euclidean metric (or any visual metric --- see
\cite{benakli-kapovich} for the definition) on $S^{1}=\partial \Gamma$.

\begin{theorem}\label{theorem:holderMartinKernel}
For simple random walk on a surface group $\Gamma =\Gamma_{g}$ of sufficiently
large genus $g$ there exists $\varrho =\varrho_{g} <1$ such that for
every $1\leq r\leq R$ and every geodesic ray $1=y_{0},y_{1}, y_{2},\dotsc$
converging to a point $\zeta \in \partial \Gamma$,
\begin{equation}\label{eq:convergenceRate}
	\Bigg| \frac{G_{r} (x,y_{n})}{G_{r} (1,y_{n})}-K_{r}
	(x,\zeta)\Bigg| 
	\leq C_{x}\varrho^{n}.
\end{equation}
The constants $C_{x}<\infty$ depend on $x\in \Gamma$ but not on $r\leq
R$. Consequently, for each $x\in \Gamma$ and $r\leq R$ the function
$\zeta \mapsto K_{r} (x,\zeta)$ is H\"{o}lder continuous (for some
positive exponent) relative to
the Euclidean metric on $S^{1}=\partial \Gamma $ in $\zeta$ for some
exponent not depending on $r\leq R$. Furthermore, the mapping $r
\mapsto K_{r} (x,\cdot )$ is continuous in the H\"{o}lder norm.
\end{theorem}

The exponential convergence \eqref{eq:convergenceRate} and the
H\"{o}lder continuity of the Martin kernel for $r=1$ were established
by Series \cite{series} for random walks on Fuchsian groups. Similar
results for the Laplace-Beltrami operator on negatively curved Cartan
manifolds were proved by Anderson and Schoen \cite{anderson-schoen}.
The methods of \cite{anderson-schoen} were adapted Ledrappier
\cite{ledrappier:review} to prove that Series' results extend to all
random walks on a free group, and Ledrappier's proof was extended by
Izumi, Neshvaev, and Okayasu \cite{izumi} to prove that for random
walk on a non-elementary hyperbolic group the Martin kernel $K_{1}
(x,\xi)$ is H\"{o}lder continuous in $\xi$. All of these proofs rest
on inequalities of the type discussed in section~\ref{ssec:anconaIneq}
below.  Theorem~\ref{theorem:1} below asserts (among other things)
that similar estimates are valid for all $G_{r}$ \emph{uniformly} for
$r\leq R$. Given these, the proof of \cite{izumi} applies almost
verbatim to establish Theorem~\ref{theorem:holderMartinKernel}:one
need only note that all estimates used in \cite{izumi} hold uniformly
for $r$ in any interval where the Ancona estimates \eqref{eq:ancona}
hold uniformly.

Routine arguments (see \cite{hamenstaedt}, Lemma~2.1, also
\cite{ledrappier:renewal}) show that Theorem~\ref{theorem:holderMartinKernel}
has the following corollary.

\begin{corollary}\label{corollary:bi}
For simple random walk on a surface group $\Gamma_{g}$ of sufficiently
large genus $g$  there is a
continuous function $\Lambda_{r} :\Gamma \times \partial \Gamma \times
\partial \Gamma \rightarrow \zz{R}_{+}$ such that for each $x\in
\Gamma$ and any two distinct points $\xi ,\zeta \in \partial \Gamma$,
if $1=y_{0},y_{1},\dotsc$ is a geodesic ray converging to $\xi$ and
$1=z_{0},z_{1},\dotsc$ a geodesic ray converging to $\zeta$, then
\begin{equation}\label{eq:bi}
	\lim_{n \rightarrow \infty} \frac{G_{r} (y_{n},x)G_{r}
	(z_{n},x)}{G_{r} (y_{n},z_{n})}=\Lambda_{r} (x;\xi ,\zeta) .
\end{equation}
The function $\Lambda_{r} (x;\xi  ,\zeta)$ vanishes when $\xi =\zeta$, and
for each $x\in \Gamma$ is jointly H\"{o}lder in $\xi ,\zeta$ relative
to a visual metric. Furthermore, for some $\varrho <1$ and constants
$C_{x,\xi ,\zeta}<\infty$ not depending on $r\leq R$,
\begin{equation}\label{eq:bi-ineq}
	\Bigg|\frac{G_{r} (y_{n},x)G_{r}
	(z_{n},x)}{G_{r} (y_{n},z_{n})}-\Lambda_{r}(x;\xi ,\zeta) \Bigg|
	\leq C_{x,\xi ,\zeta} \varrho^{n}.
\end{equation}
\end{corollary}

\subsection{Ancona's boundary Harnack inequalities}\label{ssec:anconaIneq}

The crux of Ancona's argument in \cite{ancona} was a system of
inequalities that assert, roughly, that the Green's function $G_{R}
(x,y)$ is nearly submultiplicative in the arguments $x,y\in
\Gamma$. Ancona \cite{ancona} proved that such inequalities always
hold for $r<R$: in particular, he proved, for an  symmetric
nearest neighbor random walk with finitely supported step distribution
on a hyperbolic group, that for each
$r<R$ there is a constant $C_{r}<\infty$ such that for every geodesic
segment $x_{0}x_{1}\dotsb x_{m}$ in (the Cayley graph of) $\Gamma$,
\begin{equation}\label{eq:ancona}
	G_{r} (x_{0},x_{m})\leq C_{r} G_{r} (x_{0},x_{k})G_{r}
	(x_{k},x_{m}) \qquad \forall \, 1\leq k\leq m.
\end{equation}
His argument depends in an essential way on the hypothesis $r<R$
(cf. his Condition (*)), and it leaves open the possibility that the
constants $C_{r}$ in the inequality \eqref{eq:ancona} might blow up as
$r \rightarrow R$. For finite-range random walk on a free group it can
be shown, by direct calculation, that the constants $C_{r}$ remain
bounded as $r \rightarrow R$, and that the inequalities
\eqref{eq:ancona} remain valid at $r=R$ (cf. \cite{lalley:frrw}). The
following result asserts that the same is true for random walk on the
surface group $\Gamma_{g}$ of large genus $g$.

\begin{theorem}\label{theorem:1}
For simple random walk on a surface group $\Gamma_{g}$ of sufficiently
large genus $g$,
\begin{enumerate}
\item [(A)] the Green's function $G_{R} (1,x) $ decays exponentially
in $|x|:=d (1,x)$; and
\item [(B)] Ancona's inequalities \eqref{eq:ancona} hold for all
$r\leq R$, with a constant $C$ independent of $r$.
\end{enumerate}
\end{theorem}

\begin{note}\label{note:1}
Here  and throughout the paper $d (x,y)$ denotes the distance between
the vertices $x$ and $y$ in the Cayley graph $G^{\Gamma}$,
equivalently, distance in the word metric. 
\emph{Exponential decay} of the
Green's function means \emph{uniform} exponential decay in all
directions, that is, there are constants $C<\infty$ and
$\varrho <1$ such that for all $x,y\in \Gamma_{g}$,
\begin{equation}
\label{eq:expDecay}
	G_{R} (x,y)\leq C\varrho^{d (x,y)}.
\end{equation}
A very simple argument, given in section \ref{ssec:backscattering}
below, shows that for a symmetric random walk on any nonamenable group $G_{R}
(1,x) \rightarrow 0$ as $|x| \rightarrow \infty$. Given this,
it is routine to show that exponential decay of the Green's function
follows from Ancona's inequalities. Nevertheless, an independent ---
and simpler --- proof of exponential decay is given in
section~\ref{ssec:barriers}.
\end{note}

\begin{note}\label{note:hamenstaedt}
Theorem~\ref{theorem:1} (A) is a discrete analogue of one of the main
results (Theorem B) of Hamenstaedt \cite{hamenstaedt} concerning the
Green's function of the Laplacian on the universal cover of a compact
negatively curved manifold.  Unfortunately, Hamenstaedt's proof
appears to have a serious error.\footnote{The error is in the proof of
Lemma~3.1: The claim is made that a lower bound on a finite measure
implies a lower bound for its Hausdorff-Billingsley dimension relative
to another measure. This is false -- in fact such a lower bound on
measure implies an \emph{upper} bound on its Hausdorff-Billingsley
dimension. } The approach taken here bears no resemblance to that of
\cite{hamenstaedt}.
\end{note}

Theorem ~\ref{theorem:1} is proved in section~\ref{sec:walkabout}
below. The argument uses  the \emph{planarity} of the Cayley
graph $G^{\Gamma}$ of a surface group in an essential way.
In addition,
it requires certain \emph{a priori} estimates on the Green's function,
established in section~\ref{sec:apriori}, specifically (see
Proposition~\ref{proposition:vanishingF}), that
\begin{equation}\label{eq:apriori}
	\lim_{ g \rightarrow \infty} \sup_{x\not =1}G_{R} (1,x)=0;
\end{equation}
it is here that the hypothesis of large genus is used.  The proof of
\eqref{eq:apriori} also relies on the fact that the step distribution
is uniform on the generating set $A_{g}$, together with a bound for
the inverse spectral radius $R=R_{g}$ of the simple random walk on
$\Gamma_{g}$ due to Zuk \cite{zuk} (see also Bartholdi \emph{et al}
\cite{bartholdi} and Nagnibeda \cite{nagnibeda}):
\begin{equation}\label{eq:zuk}
		R_{g}>\sqrt{g}.
\end{equation}
It is concievable that a suitable substitute for the estimate
\eqref{eq:apriori} could be established more generally, without the
symmetry hypothesis on the step distribution and without appealing to
Zuk's inequality on the spectral radius. If so, all of our results
concerning the asymptotic behavior of the Green's function would
hold at this level of generality.

\subsection{Decay at infinity of the Green's
function}\label{ssec:decayRate} Neither Ancona's result nor Theorem
\ref{theorem:1} gives any information about how the uniform
exponential decay rate $\varrho$ depends on the step distribution of
the random walk. In fact, the Green's function $G_{r} (1,x)$ decays at
different rates in different directions $x \rightarrow \partial
\Gamma$. To quantify the overall decay, consider the behavior of the
Green's function over the entire sphere $S_{m}$ of radius $m$ centered
at $1$ in the Cayley graph $G^{\Gamma}$.  If $\Gamma$ is nonelementary
and word-hyperbolic then the cardinality of the sphere $S_{m}$ grows
exponentially in $m$ (see Corollary~\ref{corollary:sphereGrowth} in
section~\ref{sec:cannon}): there exist constants $C>0$ and $\zeta >1$
such that as $m \rightarrow \infty$,
\begin{equation}\label{eq:sphereGrowthRate}
	 |S_{m}| \sim C \zeta^{m}.
\end{equation}

\begin{theorem}\label{theorem:2}
For simple random walk on a surface group $\Gamma_{g}$ of sufficiently
large genus $g$,
\begin{equation}\label{eq:backscatterA}
	 \lim_{m \rightarrow \infty}
	 \sum_{x\in S_{m}}G_{R} (1,x)^{2} = C>0
\end{equation}
exists and is finite, and 
\begin{equation}\label{eq:lp}
	\# \{x\in \Gamma \, : \, G_{R} (1,x)\geq \varepsilon \}
	\asymp
	\varepsilon^{-2}
\end{equation}
as $\varepsilon  \rightarrow 0$. (Here $\asymp$ means that the ratio
of the two sides remains bounded away from $0$ and $\infty$.)
\end{theorem}

The proof is carried out in sections~\ref{sec:gibbs}--\ref{sec:thermo}
below (cf. Propositions~\ref{proposition:gGreenAsymptotics} and
\ref{proposition:pressureEqualsZero}), using the fact that any
hyperbolic group has an \emph{automatic structure}
\cite{ghys-deLaHarpe}. The automatic structure will permit us to use
the theory of \emph{Gibbs states} and \emph{thermodynamic formalism}
of Bowen \cite{bowen}, ch.~1. Theorem~\ref{theorem:holderMartinKernel}
is essential for this, as the theory developed in \cite{bowen} applies
only to H\"{o}lder continuous functions.

 It is likely that $\asymp$ can be
replaced by $\sim$ in \eqref{eq:lp}. Note the resemblance between
relation \eqref{eq:lp} and the asymptotic formula for the number of
lattice points in the ball of radius $m$: this is no accident, because
$\log G_{R} (x,y)/G_{R}(1,1)$ is a metric on $\Gamma $ quasi-isometric
to the word metric (sec.~\ref{ssec:superadditivity} below).  There is
a simple heuristic argument that suggests why the sums $\sum_{x\in
S_{m}}G_{R} (1,x)^{2}$ should remain bounded as $m \rightarrow
\infty$: Since the random walk is $R-$transient, the contribution to
$G_{R} (1,1)<\infty$ from random walk paths that visit $S_{m}$ and
then return to $1$ is bounded (by $G_{R} (1,1)$).  For any $x\in
S_{m}$, the term $G_{R}(1,x)^{2}/G_{R}(1,1)$ is the contribution to
$G_{R} (1,1)$ from paths that visit $x$ before returning to $1$. Thus,
if $G_{R} (1,x)$ is not substantially larger than
\[
	\sum_{n=1}^{\infty} P^{1}\{X_{n}=x \; \text{and} \;\tau (m)=n\}R^{n},
\]
where $\tau (m)$ is the time of the first visit to $S_{m}$, then the
sum in \eqref{eq:backscatterA} should be of the same order of magnitude
as the total contribution to $G_{R} (1,1)<\infty$ from random walk
paths that visit $S_{m}$ and then return to $1$. Of course, the
difficulty in making this heuristic argument rigorous is that \emph{a
priori} one does not know that paths that visit $x$ are likely to be
making their first visits to $S_{m}$; it is Ancona's inequality
\eqref{eq:ancona}  that ultimately fills the gap.

\begin{note}\label{note:greenOnSphere}
A simple argument shows that for $r>1$ the sum of the Green's function on the
sphere $S_{m}$, unlike the sum of its square, explodes as $m
\rightarrow \infty$. Fix $1< r\leq R$ and $m\geq 1$. Since $X_{n}$
is transient, it will, with probability one, eventually visit the
sphere $S_{m}$. Since the steps of the random walk are of size $1$,
the minimum number of steps needed to reach $S_{m}$ is $m$. Hence,
\begin{align*}
	\sum_{x\in S_{m}}G_{r}(1,x)&=\sum_{n=m}^{\infty} \sum_{x\in
	S_{m}} P^{1}\{X_{n}=x \}r^{n}\\
	&\geq r^{m}\sum_{n=m}^{\infty} P^{1}\{X_{n}\in S_{m} \}\\
	&\geq r^{m} P^{1}\{X_{n}\in S_{m} \text{ for some}\; n\}\\
	&=r^{m}.
\end{align*}
\end{note}

\begin{note}\label{note:ledrappier}
There are some precedents for the result \eqref{eq:backscatterA}.
Ledrappier \cite{ledrappier:renewal} has shown that for Brownian motion on the
universal cover of a compact Riemannian manifold of negative
curvature, the integral of the Green's function $G_{1}
(x,y)=\int_{0}^{\infty}p_{t} (x,y)\,dt$ over the sphere $S(\varrho
,x)$ of radius $\varrho $ centered at a fixed point $x$ converges as
$\varrho \rightarrow \infty$ to a positive constant $C$ independent of
$x$. Hamenstaedt \cite{hamenstaedt} proves in the same context that
the integral of $G_{R}^{2}$ over $S (\varrho ,x)$ remains bounded as
the radius $\varrho \rightarrow \infty$. Our arguments (see
Note~\ref{note:theta} in sec.~\ref{sec:cannon}) show that for simple
random walk on a surface group of large genus the following is true:
for each value of $r$ there exists a power $1\leq \theta =\theta
(r)\leq 2$ such that
\[
	 \lim_{m \rightarrow \infty}
	 \sum_{x\in S_{m}}G_{r} (1,x)^{\theta } = C_{r}>0.
\]
\end{note}

\subsection{Critical exponent for the  Green's
function}\label{ssec:criticalExponent} 

Theorem \ref{theorem:2} implies that
the behavior of the Green's function $G_{R}(x,y)$ at the radius of
convergence as $y$ approaches the geometric boundary is intimately
related to the behavior of  $G_{r} (x,y)$ as $r\uparrow R$. The
connection between the two is rooted in the following set of
differential equations.

\begin{proposition}\label{proposition:GPrime}
\begin{equation}\label{eq:GPrime}
	\frac{d}{dr}G_{r} (x,y)=r^{-1}\sum_{z\in \Gamma}
			  G_{r} (x,z)G_{r} (z,y) -r^{-1}G_{r} (x,y)
	\quad \forall \; 0\leq r <R.
\end{equation}
\end{proposition}

Although the proof is elementary (cf. section \ref{ssec:GFs} below)
these differential equations have not (to my knowledge) been observed
before.  Theorem~\ref{theorem:2} implies that the sum in
equation~\eqref{eq:GPrime} blows up as $r \rightarrow R-$; this is
what causes the singularity of $r\mapsto G_{r} (1,1)$ at $r=R$.  The
rate at which the sum blows up determines the \emph{critical exponent}
for the Green's function, that is, the exponent $\alpha $ for which
$G_{R} (1,1) -G_{r} (1,1)\sim C (R-r)^{\alpha}$. The following theorem
asserts that the critical exponent is $1/2$.

\begin{theorem}\label{theorem:criticalExponent}
For simple random walk on a surface group $\Gamma_{g}$ of sufficiently
large genus, there exist constants $C_{x,y}>0$ such that as $r
\rightarrow R-$,
\begin{equation}\label{eq:criticalExponent}
	G_{R} (x,y)-G_{r} (x,y) \sim C_{x,y}\sqrt{R-r}.
\end{equation}
\end{theorem}

The proof of Theorem~\ref{theorem:criticalExponent} is given in
section~\ref{sec:criticalExp}. Like the proof of
Theorem~\ref{theorem:2}, it uses the existence of an automatic
structure and the attendant thermodynamic formalism. It also relies
critically on the conclusion of Theorem~\ref{theorem:2}, which
determines the value of the key thermodynamic variable.

The behavior of the generating function $G_{r} (1,1)$ in the
neighborhood of the singularity $r=R$ is of interest because it
reflects the asymptotic behavior
of the coefficients $P^{1}\{X_{n}=1 \}$ as $n \rightarrow \infty$.
Abel's theorem for power series, together with
\eqref{eq:criticalExponent}, implies that if there are constants
$C,\alpha>0$ such that  
\begin{equation}\label{eq:llt}
	P^{1}\{X_{2n}=1 \}\sim \frac{C}{R^{2n}n^\alpha }
	\quad \text{as} \quad n \rightarrow \infty 
\end{equation}
then $\alpha =3/2$. Unfortunately, to deduce a limit theorem of the
type \eqref{eq:llt} (with $\alpha =3/2$) from
\eqref{eq:criticalExponent} one must verify an additional Tauberian
hypothesis of some sort. For instance, if it could be shown that the
relation \eqref{eq:criticalExponent} extends off the real axis to a
neighborhood of $z=R$ in the slit plane $\zz{C}\setminus [R,\infty)$
then a Tauberian theorem of Flajolet and Odlyzko
\cite{flajolet-odlyzko} together with \eqref{eq:criticalExponent}
would imply that the return probabilities satisfy \eqref{eq:llt} with
$\alpha =3/2$.  It seems likely that there is such an off-axis
extension, because the Green's function itself has such an analytic
continuation. (Recall that the Green's function is the kernel of the
resolvent operator $r^{-1} (r^{-1}I-\zz{P})$; since $\zz{P}$ is
Hermitian, its spectrum lies entirely on the real axis.) However, the
methods developed here to establish \eqref{eq:criticalExponent} will
use the positivity of the Green's function for real arguments $r\in
(0,R]$ is an essential way; a proof that \eqref{eq:criticalExponent}
has an analytic continuation will require new methods.

Local limit theorems of the form \eqref{eq:llt} have been established
for random walks on free groups \cite{gerl-woess}, \cite{lalley:frrw},
certain free products \cite{woess:freeProducts}, and certain virtually
free groups, including $SL_{2} (\zz{Z})$ \cite{lalley:rwrl},
\cite{nagnibeda-woess}. In all of these cases the Green's function is
an algebraic function of $r$, and so verification of the hypotheses of
the Flajolet-Odlyzko Tauberian theorem is trivial, given the behavior
\eqref{eq:criticalExponent} on the real axis. In all likelihood the
Green's function of simple random walk on a surface group is not
algebraic.

\subsection{Standing Conventions}\label{ssec:convention}
The values of constants $C$, $C_{x}$, and so on may change from line
to line. The symbol $R$ is reserved for the  radius of convergence of the
Green's function. The Green's function will be denoted by $G_{r} (x,y)$
throughout, but the symbol $G$ is also used with superscript $\Gamma$
to denote the Cayley graph of $\Gamma$. The symbol $\sim$ is used in
the conventional way, meaning that the ratio of the two sides
approaches $1$.

\section{Green's function: preliminaries}\label{sec:preliminaries}

Throughout this section, $X_{n}$ is a symmetric, nearest neighbor
random walk on a finitely generated, nonamenable group $\Gamma$ with
(symmetric) generating set $A$. 

\subsection{Green's function as a sum over paths}\label{ssec:GFs}

The Green's function $G_{r} (x,y)$ defined by \eqref{eq:green} has an
obvious interpretation as a sum over paths from $x$ to $y$. (Note:
Here and in the sequel a \emph{path} in $\Gamma$ is just the sequence
of vertices visited by a path in the Cayley graph $G^{\Gamma}$, that
is, a sequence of group elements such that any two successive elements
differ by right-multiplication by a generator $a\in A$.) Denote by
$\mathcal{R} (x,y)$ the set of all paths $\gamma $ from $x$ to $y$,
and for any such path $\gamma = (x_{0},x_{1},\dotsc ,x_{m})$ define
the \emph{weight}
\begin{equation}\label{eq:weight}
	w_{r} (\gamma):=r^{m}\prod_{i=0}^{m-1}p (x_{i},x_{i+1}).
\end{equation}
Then 
\begin{equation}\label{eq:greenByPath}
	G_{r} (x,y)=\sum_{\gamma \in \mathcal{R} (x,y)}w_{r} (\gamma).
\end{equation}
Since the step distribution $p (a)=p (a^{-1})$ is symmetric with
respect to inversion, so is the weight function $\gamma \mapsto w_{r}
(\gamma)$: if $\gamma^{R}$ is the reversal of the path $\gamma$, then
$w_{r} (\gamma^{R})=w_{r} (\gamma)$. Consequently, the Green's
function is symmetric in its arguments:
\begin{equation}\label{eq:GreenSymmetry}
	G_{r} (x,y)=G_{r} (y,x).
\end{equation}
 Also, the weight function is
multiplicative with respect to concatenation of paths, that is, $w_{r}
(\gamma \gamma ')=w_{r} (\gamma)w_{r} (\gamma ')$. Since the step
distribution $p (a)>0$ is strictly positive on the generating
set $A$, it follows that the Green's function satisfies a system of
\emph{Harnack inequalities}: There exists a constant $C<\infty$ such
that for each $0<r\leq R$ and all group elements $x,y,z$,
\begin{equation}\label{eq:harnack}
	G_{r} (x,z)\leq C^{d (y,z)}G_{r} (x,y).
\end{equation}

\begin{proof}[Proof of Proposition~\ref{proposition:GPrime}]
This is a routine calculation based on the representation
\eqref{eq:greenByPath} of the Green's function as a sum over paths.
Since all terms in the power series representation of the Green's
function have nonnegative coefficients, interchange of $d/dr$ and
$\sum_{\gamma}$ is permissible, so
\[
	\frac{d}{dr}G_{r} (x,y)=\sum_{\gamma \in \mathcal{R} (x,y)}
	\frac{d}{dr} w_{r} (\gamma).
\]
  If $\gamma$ is a path from $1$ to $x$
of length $m$, then the derivative with respect to $r$ of the weight
$w_{r} (\gamma)$ is $mw_{r} (\gamma)/r$, so $dw_{r} (\gamma)/dr$
contributes one term of size $w_{r} (\gamma)/r$ for each vertex
visited by $\gamma$ after its first step.  This, together with the
multiplicativity of $w_{r}$, yields the identity \eqref{eq:GPrime}.
\end{proof}

\subsection{First-passage generating functions}\label{ssesc:fpGF}
Other useful generating functions can be obtained by summing path
weights over different sets of paths. Two classes of such generating
functions that will be used below are the \emph{restricted Green's
functions} and the \emph{first-passage} generating functions (called
the \emph{balayage} by Ancona \cite{ancona}) defined as follows. Fix a
region $\Omega \subset G^{\Gamma}$ (or alternatively a set f vertices
$\Omega \subset \Gamma $), and for any two vertices $x,y\in
G^{\Gamma}$ let $\mathcal{P}(x,y;\Omega)$ be the set of all paths from
$x$ to $y$ that remain in the region $\Omega$ at all except the
initial and final points. Define
\begin{align}\label{eq:fpgf}
	G_{r} (x,y;\Omega)&=\sum_{\mathcal{P} (x,y;\Omega )}w_{r}
	(\gamma),\quad \text{and}\\
\notag 	F_{r} (x,y)&=G_{r} (x,y;\Gamma \setminus \{y \}). 
\end{align}
Thus, $F_{r} (x,y)$, the \emph{first-passage generating function}, is
the sum over all paths from $x$ to $y$ that first visit $y$ on the
last step. This generating function has the alternative representation
\begin{equation}\label{eq:fpExpForm}
	F_{r} (x,y)=E^{x}r^{\tau (y)}
\end{equation}
where $\tau (y)$ is the time of the first visit to $y$ by the random
walk $X_{n}$, and the expectation extends only over those sample paths
such that $\tau (y)<\infty$.  Note that the restricted Green's
functions $G_{r} (\cdot,\cdot ;\Omega)$ obey Harnack inequalities
similar to \eqref{eq:harnack}, but with the distance $d (y,z)$
replaced by the distance $d_{\Omega} (y,z)$ in the set $\Omega$.
Finally, since any visit to $y$ by a path started at $x$ must follow a
\emph{first} visit to $y$,
\begin{equation}\label{eq:GxFxG}
	G_{r} (x,y)=F_{r} (x,y)G_{r} (1,1).
\end{equation}
Therefore, since $G_{r}$ is symmetric in its arguments, so is $F_{r}$.

\subsection{Renewal equation for the Green's function}\label{ssec:re}
The representation \eqref{eq:green} suggests that $G_{r}(1,1)$ can be
interpreted as the expected ``discounted'' number of visits to the
root $1$ by the random walk, where the discount factor is $r$. Any
such visit must either occur at time $n=0$ or after the first step,
which must be to a generator $x\in A$. Conditioning on the first step
and using the Markov property, together with the symmetry $F_{r}
(1,x)=F_{r} (x,1)$, yields the \emph{renewal equation}
\[
	G_{r} (1,1)=1+\sum_{x\in A} p_{x} r F_{r} (1,x) G_{r} (1,1),
\]
which may be rewritten in the form
\begin{equation}\label{eq:re}
	G_{r} (1,1)=1\big/\left(1- \sum_{x\in A} p_{x}r F_{r} (1,x)
	\right).
\end{equation}
Since $G_{R} (1,1)<\infty$ (recall that the group $\Gamma$ is
nonamenable), it follows that 
\begin{equation}\label{eq:Frestriction}
	\sum_{x\in A} p_{x}R F_{r} (1,x) <1.
\end{equation}

\subsection{Retracing inequality}\label{ssec:retracing} The
first-passage generating function $F_{r} (1,x)$ is the sum of weights
of all paths that first reach $x$ at the last step. For $x\in A$ this
may occur in one of two ways: either the path jumps from $1$ to $x$ at
its first step, or it first jumps to some $y\not =x$ and then later
finds its way to $x$. The latter will occur if the path returns to the
root $1$ from $y$ without visiting $x$, and then finds its way from
$1$ to $x$. This leads to a simple bound for $F_{r} (1,x)$ in terms of
the \emph{avoidance generating function} $A_{r} (1;x)$, defined  by
\begin{equation}\label{eq:avoidance}
A_{r} (1;x):=\sum_{y\not =x} p_{y}R G_{r} (y,1;\Gamma \setminus \{x,1 \})
	      =G_{r} (1,1;\Gamma   \setminus\{x,1 \}).
\end{equation}
Thus, $A_{r} (1,x)$ is the sum of weights $w_{r} (\gamma)$ over all
paths $\gamma$ that begin and end at $1$, and avoid both $1$ and $x$
in transit.

\begin{lemma}\label{lemma:avoidance} 
The avoidance generating function satisfies $A_{R} (1;x)<1$  for
every $x\in A$, and for every $r\leq R$,
\begin{equation}\label{eq:retracing}
	F_{r} (1,x)\geq p_{x}R / ( 1-A_{r} (1;x)).
\end{equation}
For simple random walk on the surface group $\Gamma_{g}$, the
inequality is strict.
\end{lemma}

\begin{proof}
A path $\gamma$ that starts at the root $1$ can reach $x$ by jumping
directly from $1$ to $x$, on the first step, or by jumping from $1$ to
$x$ after an arbitrary number $n\geq 1$ of returns to $1$ without
first visiting $x$. Hence,
\[
	F_{r} (1,x)\geq p_{x}R \left\{ 1+\sum_{n=1}^{\infty} A_{r} (1;x)^{n}\right\}.
\]
The inequality is strict for random walk on the surface group because
in this case there are positive-probability paths from $1$ to $x$
that do not end in a jump from $1$ to $x$. Since $F_{R} (1,x)<\infty$,
by \eqref{eq:GxFxG} and \eqref{eq:finiteAtR}, it must be that $A_{R}
(1;x)<\infty$. 
\end{proof}

\subsection{Renewal inequality}\label{ssec:renewalIneq}
The avoidance generating functions can be used to reformulate the
renewal equation \eqref{eq:re} in a way that leads to a useful upper
bound for the Green's function. Recall that the  renewal equation was
obtained by splitting paths that return to the root $1$ at the time of
their \emph{first} return. Consider a path $\gamma$ starting at $1$
that first returns to $1$ only at its last step: such a path must
either avoid $x\in A$ altogether, or it must visit $x$ before the
first return to $1$, and then subsequently find its way back to
$1$. Thus,
for any generator $x\in A$,
\begin{align*}
	G_{r} (1,1)&=1+A_{r} (1;x)G_{r} (1,1)+
	 F_{r} (1,x;\Omega \setminus \{1 \})F_{r} (x,1)G_{r} (1,1)\\
	 &\leq 1+A_{r} (1;x)G_{r} (1,1)+F_{r} (x,1)^{2}G_{r} (1,1).
\end{align*}
Solving for $G_{r} (1,1)$ gives the following \emph{renewal inequality}:

\begin{equation}\label{eq:renewalIneq}
	G_{r} (1,1)\leq \{1-A_{r} (1;x)-F_{r} (1,x)^{2} \}^{-1}.
\end{equation}

\subsection{Backscattering}\label{ssec:backscattering}
A very simple argument shows that the Green's function $G_{R} (1,x)$
converges to $0$ as $|x| \rightarrow \infty$. Observe that if $\gamma$
is a path from $1$ to $x$, and $\gamma '$ a path from $x$ to $1$, then
the concatenation $\gamma \gamma '$ is a path from $1$ back to
$1$. Furthermore, since any path from $1$ to $x$ or back  must make at
least $|x|$ steps, the length of $\gamma \gamma '$ is at least
$2|x|$. Consequently, by symmetry,
\begin{equation}\label{eq:back-forth}
	F_{R} (1,x)^{2}G_{R} (1,1)\leq \sum_{n=2|x|}^{\infty}
	      P^{1}\{X_{n}=1 \}R^{n}
\end{equation}
Since $G_{R} (1,1)<\infty$, by nonamenability of the group $\Gamma$,
the tail-sum on the right side of inequality \eqref{eq:back-forth}
converges to $0$ as $|x| \rightarrow \infty$, and so $F_{R}
(1,x)\rightarrow 0$ as $|x| \rightarrow \infty$. Several variations on
this argument will be used later.

\subsection{Subadditivity and the random walk
metric}\label{ssec:superadditivity} The concatenation of a path from
$x$ to $y$ with a path from $y$ to $z$ is, obviously, a path from $x$
to $z$. Consequently, by the Markov property (or alternatively the
path representation \eqref{eq:greenByPath} and the multiplicativity of
the weight function $w_{r}$) the function $-\log F_{r} (x,y)$ is
\emph{subadditive}:

\begin{lemma}\label{lemma:superadditivity}
For each $r\leq R$ the first-passage generating functions  $F_{r}
(x,y)$  are 
super-- multiplicative, that is, for any group elements $x,y,z$,
\begin{equation}\label{eq:GreenSuperM}
	F_{r} (x,z)\geq F_{r} (x,y)F_{r} (y,z) .
\end{equation}
\end{lemma}

Together with Kingman's subadditive ergodic theorem, this implies that
the Green's function $G_{r} (1,x)$ must decay  at a fixed
exponential rate along suitably chosen trajectories. For instance, if
\begin{equation}\label{eq:markovProduct}
	 Y_{n}=\xi_{1}\xi_{2}\dotsb \xi_{n}
\end{equation}
where $\xi_{n}$ is an ergodic Markov chain on the alphabet $A$, or on
the set $A^{K}$ of words of length $K$, then Kingman's theorem implies
that 
\begin{equation}\label{eq:kingmanRW}
	\lim n^{-1}\log G_{r} (1,Y_{n}) =\alpha  \quad \text{a.s.}
\end{equation}
where $\alpha$ is a constant depending only on the the transition
probabilities of the underlying Markov chain. More generally, if
$\xi_{n}$ is a suitable ergodic stationary process, then
\eqref{eq:kingmanRW} will hold. 
Super-multiplicativity of the
Green's function also implies the following.
 
\begin{corollary}\label{corollary:RWMetric}
The function $d_{G} (x,y):=\log F_{R} (x,y)$ is a metric on $\Gamma$.
\end{corollary}

\begin{proof}
The triangle inequality is immediate from
Lemma~\ref{lemma:superadditivity}, and  symmetry
$d_{G}(x,y)=d_{G} (y,x)$  follows from the corresponding symmetry
property \eqref{eq:GreenSymmetry} of the Green's function. Thus, to
show that $d_{G}$ is a metric (and not merely a pseudo-metric) it
suffices to show that if $x\not =y$ then $F_{R} (x,y)<1$.  But this
follows from the fact \eqref{eq:finiteAtR} that the Green's function
is finite at the spectral radius, because the path representation
implies that
\[
	G_{R} (x,x)\geq 1+F_{R} (x,y)^{2} +F_{R} (x,y)^{4} +\dotsb .
\]
\end{proof}

Call  $d_{G}$  the \emph{Green
metric}.  The Harnack inequalities imply that 
$d_{G}$ is dominated by a constant multiple of the word metric
$d$. In general, there is no domination in the other
direction. However:

\begin{proposition}\label{proposition:qi}
If the Green's function decays exponentially in $d (x,y)$ (that is, if
inequality \eqref{eq:expDecay} holds for all $x,y\in \Gamma$), then the
Green metric $d_{G}$ and the word metric $d$ on ${\Gamma}$
are quasi-isometric, that is, there are constants
$0<C_{1}<C_{2}<\infty$ such that for all $x,y\in \Gamma $,
\begin{equation}\label{eq:qi}
	C_{1}d (x,y)\leq d_{G} (x,y)\leq C_{2}d (x,y).
\end{equation}
\end{proposition}

\begin{proof}
If inequality \eqref{eq:expDecay} holds for all $x,y\in \Gamma$, then
the first inequality in \eqref{eq:qi} will hold with  $C_{1}=-\log \varrho$.
\end{proof}

\begin{note}\label{remark:metrics}
Except in the simplest cases --- when the Cayley graph is a tree, as
for free groups and free products of cyclic groups --- the random walk
metric $d_{G}$ does \emph{not} extend from $\Gamma $ to a metric on
the full Cayley graph $G^{\Gamma}$. To see this, observe that if it
\emph{did} extend, then the resulting metric space
$(G^{\Gamma},d_{G})$ would be path-connected, and therefore
would have the Hopf-Rinow property: any two points would be connected
by a geodesic segment. But this would imply that for vertices $x,z\in
\Gamma$ such that $d (x,z)\geq 2$ there would be a $d_{G}-$ geodesic segment
from $x$ to $z$, and such a geodesic would necessarily pass through a
point $y$ such that $d (x,y)=1$. For any such triple $x,y,z\in \Gamma$
it would then necessarily be the case that 
\[
	F_{R} (x,z)=F_{R} (x,y)F_{R} (y,z).
\]
This is possible only when every random walk path from $x$ to $z$
must pass through $y$ --- in particular, when the Cayley graph of
$\Gamma$ is a tree.
\end{note}

\subsection{Green's function and branching random
walks}\label{ssec:brw} There is a simple interpretation of the Green's
function $G_{r} (x,y)$ in terms of the occupation statistics of
\emph{branching random walks}. A branching random walk is built using
a probability distribution $\mathcal{Q}=\{q_{k} \}_{k\geq 0}$ on the
nonnegative integers, called the \emph{offspring distribution},
together with the step distribution $\mathcal{P}:=\{ p(x,y)=p
(x^{-1}y)\}_{x,y\in \Gamma }$ of the 
underlying random walk, according to the following rules: At each time
$n\geq 0$, each particle fissions and then dies, creating a random
number of offspring with distribution $\mathcal{Q}$; the offspring
counts for different particles are mutually independent. Each
offspring particle then moves from the location of its parent by
making a random jump according to the step distribution $p (x,y)$; the
jumps are once again mutually independent. Consider the initial
condition which places a single particle at site $x\in \Gamma$, and
denote the corresponding probability measure on population evolutions
by $Q^{x}$.

\begin{proposition}\label{proposition:brw}
Under $Q^{x}$, the total number of particles in generation $n$ evolves
as a Galton-Watson process with offspring distribution
$\mathcal{Q}$. If the offspring distribution has mean $r\leq R$, then
under $Q^{x}$ the expected number of particles at location $y$ at time
$n$ is $r^{n}P^{x}\{X_{n}=y \}$, where under $P^{x}$ the process
$X_{n}$ is an ordinary random walk with step distribution
$\mathcal{P}$. Therefore, $G_{r} (x,y)$ is the mean total number of
particle visits to location $y$.
\end{proposition}

\begin{proof}
The first assertion follows easily from the definition of a
Galton-Watson process -- see \cite{athreya-ney} for the definition and
basic theory.  The second is easily proved by induction on $n$. The third
then follows from the formula \eqref{eq:green} for the Green's
function.
\end{proof}

There are similar interpretations of the restricted Green's function
$G_{r} (x,y;\Omega)$ and the first-passage generating function $F_{r}
(x,y)$. Suppose that particles of the branching random walk are
allowed to reproduce only in the region $\Omega$; then
$G_{r}(x,y;\Omega)$ is the mean number of particle visits to $y$ in
this modified branching random walk.

\section{{A priori} estimates for the Green's function}\label{sec:apriori}



\subsection{Symmetries of simple random walk on
$\Gamma_{g}$}\label{ssec:symmetries} 
Recall that the generating set $A_{g}$ of the surface group
$\Gamma_{g}$ consists of $2g$ letters $a_{i},b_{i}$ and their
inverses, which are subject to the relation $ \prod [a_{i},b_{i}] =1$.
This fundamental relation implies others, including 
\begin{gather}\label{eq:reversal}
	\prod_{i=0}^{g-1} [b_{g-i},a_{g-i}]=1 \qquad \text{and}\\
\label{eq:cycle}
	\prod_{i=k+1}^{g}[a_{i},b_{i}] \prod_{i=1}^{k} [a_{i}^{-1},b_{i}^{-1}]=1.
\end{gather}
Since each of these has the same form as the fundamental relation,
each leads to an automorphism of the group $\Gamma_{g}$: relation
\eqref{eq:reversal} implies that the bijection 
\begin{align*}
	a_{i}^{\pm 1}&\mapsto b_{g-i}^{\pm 1},\\
	b_{i}^{\pm 1}&\mapsto a_{g-i}^{\pm 1}
\end{align*}
extends to an automorphism, and similarly relation \eqref{eq:cycle}
implies that  the mapping
\begin{align*}
	a_{i}^{\pm 1} & \mapsto a_{i+1}^{\pm 1} \quad 1\leq i\leq g-1, \\
	b_{i}^{\pm 1} & \mapsto b_{i+1}^{\pm 1} \quad 1\leq i\leq g-1, \\
	a_{g}^{\pm 1}& \mapsto a_{1}^{\mp 1},\\
	b_{g}^{\pm 1}& \mapsto  b_{1}^{\mp 1}
\end{align*}
extends to an automorphism. Clearly, each of these automorphisms
preserves the uniform distribution on $A_{g}$, and so it must also fix each of the
generating functions $G_{r} (1,x)$, $ F_{r} (1,x)$, and $A_{r}(1;x)$.
This implies

\begin{corollary}\label{corollary:automorphisms}
For simple random walk on $\Gamma_{g}$,
\begin{align}\label{eq:autA}
		A_{r} (1;x)&=A_{r} (1;y):=A_{r} \quad \forall \; x,y\in A_{g},\\
\notag 	F_{r} (1,x)&=F_{r} (1,y):=F_{r} \quad \forall \; x,y\in A_{g}, \quad \text{and}\\
\notag 	G_{r} (1,x)&=G_{r} (1,y):=G_{r} \quad \forall \; x,y\in A_{g} .	
\end{align}
\end{corollary}

Consequently, by the renewal equation, $G_{r}=1/ (1-rF_{r})$. Since
$G_{R}<\infty$, this implies that for the simple random walk on
$\Gamma_{g}$ the first-passage generating functions $F_{r} (1,x)$ for
$x\in A_{g}$ are bounded by $1/R$:
\begin{equation}\label{eq:crudeFBound}
		F_{r}\leq F_{R}< 1/R.
\end{equation}

\subsection{Large genus asymptotics for $G_{R}
(1,1)$}\label{ssec:largegAsym}
For simple random walk on $\Gamma_{g}$, the symmetry relations
\eqref{eq:autA} and the inequalities \eqref{eq:retracing},
\eqref{eq:renewalIneq}, and \eqref{eq:zuk} can be combined to give
upper bounds for the Green's function. Asymptotically, these take the
following form:

\begin{proposition}\label{proposition:gGreenAsymptotics}
$\lim_{g \rightarrow \infty}G_{R} (1,1)=2$.
\end{proposition}

\begin{proof}
Proposition \ref{proposition:catalan} below implies that $\liminf \geq
2$, so it suffices to prove the reverse inequality $\limsup \leq 2$. For
notational convenience, write $G=G_{R} (1,1)$, $F=F_{R} (1,x)$, and
$A=A_{R} (1;x)$; by Corollary \ref{corollary:automorphisms}, the
latter two quantities do not depend on the generator $x$. As noted
above, the renewal equation implies that 
\[
	 RF= 1-1/G,
\]
and by \eqref{eq:crudeFBound} above,  $F<1/R$. By Zuk's
inequality \eqref{eq:zuk}, the spectral radius $R=R_{g}$ is at least
$\sqrt{g}$, so it follows that $F<1/\sqrt{g}$, which is asymptotically
negligible as the  genus $g \rightarrow \infty$.
On the other hand, the retracing inequality \eqref{eq:retracing},
together with Zuk's inequality, gives
\[
	1-1/G=RF> R^{2} / (4g (1-A))>1/ (4 (1-A)).
\]
This implies that $A<1$. In the other direction, the renewal
{inequality} \eqref{eq:renewalIneq} implies that
\[
	1/G \geq (1-A-F^{2}).
\]
Combining the last two inequalities yields
\[
	1/4 (1-A)<A+F^{2}
\]
Since $F \rightarrow 0$ as $g \rightarrow \infty$, it follows that
$\liminf_{g \rightarrow \infty}A\geq 1/2$. Finally, using once again
the renewal inequality \eqref{eq:renewalIneq} and the fact that $F$ is
asymptotically negligible as $g \rightarrow \infty$, 
\begin{equation}\label{eq:gUpperAsymptotics}
	\limsup_{g \rightarrow \infty} G_{R} (1,1)\leq 2.
\end{equation}
\end{proof}

\subsection{The covering random walk}\label{ssec:catalan} The simple
random walk $X_{n}$ on the surface group $\Gamma_{g}$ can be lifted in
an obvious way to a simple random walk $\tilde{X}_{n}$, called the
\emph{covering random walk}, on the free group $\mathcal{F}_{2g}$ on
$2g$ generators. Clearly, on the event $\tilde{X}_{2n}=1$ that the
lifted walk returns to the root at time $2n$, it must be the case that
the projection $X_{2n}=1$ in $\Gamma_{g}$. Thus, the return
probabilities for $X_{2n}$ are bounded below by those of
$\tilde{X}_{2n}$, and so the spectral radius $R$ is bounded above by
the spectral radius $\tilde{R}$ of the covering random walk.  The
return probabilities of the covering random walk are easy to
estimate. Each step of $\tilde{X}_{n}$ either increases or decreases
the distance from the root $1$ by $1$; the probability that the
distance increases is $(4g-1)/4g$, unless the walker is at the root,
in which case the probability that the distance increases is
$1$. Consequently, the probability that $\tilde{X_{2n}}=1$ can be
estimated from below by counting up/down paths of length $2n$ in the
nonnegative integers that begin and end at $0$. It is well known that
the number of such paths is the $n$th \emph{Catalan number}
$\kappa_{n}$. Thus,
\begin{equation}\label{eq:catalan}
	P^{1}\{X_{2n}=1 \} \geq P^{1}\{\tilde{X}_{2n}=1 \}
			\geq \kappa_{n} \left(\frac{4g-1}{4g}
			\right)^{n} \left(\frac{1}{4g} \right)^{n}.
\end{equation}
The generating function of the Catalan numbers is
\[
	\sum_{n=0}^{\infty} \kappa_{n}z^{n} = \frac{1-\sqrt{1-4z}}{2z}.
\]
The smallest positive singularity is at $z=1/4$, and the value of the
sum at this argument is $2$. Since the spectral radius satisfies
$R_{g}^{2}/4g \rightarrow 1/4$ as $g \rightarrow \infty$, by Zuk's
inequality and results of Kesten \cite{kesten2}, the inequality
\eqref{eq:catalan} has the following consequence.

\begin{proposition}\label{proposition:catalan}
For every $\varepsilon >0$ there exist $g ({\varepsilon})<\infty $ and
$m ({\varepsilon})<\infty$ such that if $g\geq g (\varepsilon)$ then
\begin{equation}\label{eq:uniformTailBound}
	\sum_{n=0}^{m (\varepsilon)} P^{1}\{X_{2n}=1 \}R^{2n}
	\geq 2-\varepsilon.
\end{equation}
\end{proposition}

The Green's function and first-passage generating functions for the
covering random walk can be exhibited in closed form, using the
renewal equation and a retracing identity. (See \cite{woess:book} for
much more general results.) The key is that the Cayley
graph of the free group is the infinite homogeneous tree $\zz{T}_{4g}$
of degree $4g$. Since there are no cycles, for any two distinct
vertices $x,y\in \Gamma_{g}$ there is only one self-avoiding path (and
hence only one geodesic segment) from
$x$ to $y$; therefore, if $x=x_{1}x_{2}\dotsb x_{m}$ is the word
representation of $x$ then
\begin{equation}\label{eq:FreeFProduct}
	\tilde{F}_{r} (1,x)=\prod_{i=1}^{m}\tilde{F}_{r} (1,x_{i}) =\tilde{F}_{r}^{|x|},
\end{equation}
the last because symmetry forces $\tilde{F}_{r} (1,y)=\tilde{F}_{r}$
to have a common value for all generators $y$. Now fix a generator
$x\in A_{g}$, and consider the first-passage generating function
$\tilde{F}_{r} (1,x)=\tilde{F}_{r}$: Since $\zz{T}_{4g}$ has no
cycles, any random walk path from $1$ to $x$ must either jump directly
from $1$ to $x$, or must first jump to a generator $y\not =x$, then
return to $1$, and then eventually find its way to $x$. Consequently,
with $p=p_{g}=1-q=1/4g$,
\[
	\tilde{F}_{r}=pr +qr\tilde{F}_{r}^{2} ,
\]
from which it follows that 
\begin{align}\label{eq:coveringF}
	\tilde{F}_{r}&=\frac{1-\sqrt{1-4pqr^{2}}}{2qr},\\
\notag 
	\tilde{G}_{r}&=\frac{2q}{2q-1-\sqrt{1-4pqr^{2}}},  \quad \text{and}\\
\notag 
	\tilde{R}^{2}&=\frac{1}{4pq}=\frac{4g^{2}}{4g-1}.
\end{align}

\subsection{Uniform bounds on the Green's
function}\label{ssec:uniformBoundsGreen} 
Recall from section \ref{ssec:backscattering} that the first-passage
generating function $F_{R} (1,x)$ is bounded by the tail-sums of the
Green's function $G_{R} (1,1)$: in particular,
\begin{equation}\label{eq:Fbound}
	F_{R} (1,x)^{2}\leq F_{R} (1,x)^{2}G_{R} (1,1)
	\leq \sum_{n=2 |x|}^{\infty} P^{1}\{X_{n}=1 \}R^{n}.
\end{equation}
 Propositions
\ref{proposition:gGreenAsymptotics}--\ref{proposition:catalan} imply
that, for large genus $g$, these tail-sums can be made uniformly small by
taking $|x|$ sufficiently large. In fact,  the first-passage
generating functions can be bounded away from $1$ uniformly in $x\in
\Gamma_{g}\setminus \{1 \}$ provided the genus is sufficiently large:

\begin{proposition}\label{proposition:uniformBoundsF}
For any $\alpha >3/4$ 
there exists $g_{\alpha }<\infty$ so that
\begin{equation}\label{eq:uniformBoundsF}
		\sup_{g\geq g_{\alpha }} \sup_{x\not =1}
		F_{R} (1,x) <\sqrt{\alpha }.
\end{equation}
\end{proposition}

\begin{proof}
By Proposition \ref{proposition:gGreenAsymptotics}, $G_{R} (1,1)$ is
close to $2$ for large genus $g$. On the other hand, by inequality
\eqref{eq:catalan}, $P^{1}\{X_{2}=1 \}R^{2} \geq (1-1/4g) (R^{2}/4g)$,
and by taking $g$ large this can be made arbitrarily close to
$1/4$. Hence, by taking $g\geq g_{*}$ with $g_{*}$ large,
\[
	\sum_{n=3}^{\infty} P^{1}\{X_{n}=1 \}R^{n}\leq \alpha 
\]
where $\alpha$ can be taken arbitrarily close to $3/4$ by letting
$g_{*} \rightarrow \infty$.   Inequality
\eqref{eq:Fbound} now implies that $F_{R} (1,x)^{2}\leq \alpha$ for
all $|x|\geq 2$ and all $g\geq g_{*}$. But for $|x|=1$, the symmetry
relations of Corollary~\ref{corollary:automorphisms} and the renewal
equation \eqref{eq:re} imply that $F_{R} (1,x)<1/R$, which tends to
$0$ as $g \rightarrow  \infty$. 
\end{proof}

A more sophisticated version of this argument shows

\begin{proposition}\label{proposition:vanishingF}
\begin{equation}\label{eq:vanishingF}
	\lim_{g \rightarrow \infty} \sup_{x\not =1} F_{R} (1,x)=0.
\end{equation}
\end{proposition}

\begin{proof}
First, consider a vertex $x\in \Gamma_{g}$ at distance 
$\geq 2g$ from the root $1$: By inequality \eqref{eq:Fbound}, $F_{R}
(1,x)^{2}$  is bounded by the tail-sum $\sum_{n\geq 4g} P\{X_{n}=1
\}R^{n}$, which by  Propositions \ref{proposition:gGreenAsymptotics} 
--  \ref{proposition:catalan} converges to zero as $g
\rightarrow \infty$. Thus, it remains only to show that $F_{R}
(1;x)\rightarrow 0$ as $g \rightarrow \infty$ uniformly for vertices
$x$ at distance $< 2g$ from the root. 

Fix a vertex $x\in \Gamma_{g}$ such that $|x|< 2g$, and consider a
path $\gamma$ from $1$ to $x$. If $\gamma $ is of length $< 2g$
then it has no nontrivial cycles, because the  fundamental relation
\eqref{eq:surfaceRelations}  has length $4g$. Consequently, it lifts
to a path $\tilde{\gamma}$ in the free group $\mathcal{F}_{2g}$ from
$1$ to the unique  covering point $\tilde{x}$ of $x$ at distance $<2g$
from the group identity $1$ in $\mathcal{F}_{2g}$. Since $R\leq
\tilde{R}$, it follows that 
\[
	\sum_{\gamma :|\gamma |< 2g} w_{R} (\gamma)\leq \tilde{F}_{{R}} (1,\tilde{x})
	\leq \tilde{F}_{\tilde{R}} (1,\tilde{x})=\tilde{F}_{\tilde{R}}^{|x|}
\]
where the sum is over all paths from $1$ to $x$ of length $< 2g$. By
\eqref{eq:coveringF}, this converges to $0$ as $g \rightarrow \infty$.
On the other hand, since the concatenation of a path $\gamma$ from $1$
to $x$ of length $\geq 2g$ with a path $\gamma '$ of length $\geq 2g$
from $x$ to $1$ is a path from $1$ to $1$ of length $\geq 4g$,
\[
	\left(\sum_{\gamma :|\gamma |\geq  2g} w_{R} (\gamma)
	\right)^{2}
	\leq 
	\sum_{n=4g}^{\infty} P^{1}\{X_{n}=1 \}R^{n};
\]
here the sum is over all paths from $1$ to $x$ of length $\geq 2g$.
By Propositions
\ref{proposition:catalan}--\ref{proposition:uniformBoundsF}, the
tail-sum on the right side converges to zero as $g \rightarrow
\infty$.
\end{proof}

\section{The Ancona Inequalities}\label{sec:walkabout} From the
viewpoint of a random walker, the hyperbolic plane is a vast outback;
failure to follow, or nearly follow, the geodesic path from one point
to another necessitates walkabouts whose extents grow
\emph{exponentially} with the deviation from the geodesic path. It is
this that accounts for Ancona's inequality. The arguments of this
section make this precise.

\subsection{Free subgroups and embedded
trees}\label{ssec:embeddedTrees}
Recall that the surface group $\Gamma_{g}$ in its standard presentation has $2g$
generators which, together with their inverses, satisfy the relation
\eqref{eq:surfaceRelations}.  Denote by $\mathcal{F}_{A}^{+}$ and
$\mathcal{F}_{A}^{-}$ the sub-semigroups of $\Gamma_{g}$ generated by
$\{a_{i} \}_{i\leq g}$ and $\{a_{i}^{-1} \}_{i\leq g}$, respectively,
and define $\mathcal{F}_{B}^{\pm}$ similarly. 

\begin{proposition}\label{proposition:embeddedTrees}
The image of each of the semigroups $\mathcal{F}_{A}^{\pm}$ and
$\mathcal{F}_{B}^{\pm}$ in the Cayley graph is a rooted tree of
outdegree $g$. Every self-avoiding path in (the image of) any one of
these semigroups is a geodesic in the Cayley graph.
\end{proposition}

\begin{proof}
This is an elementary consequence of Dehn's algorithm
(cf. \cite{stillwell}). Consider a self-avoiding path $\alpha
=a_{(1)}a_{i (2)}\dotsb a_{i (m)}$ in $\mathcal{F}_{A}^{+}$. If this
were not a geodesic segment, then there would exist a geodesic path 
$\beta =x_{n}^{-1}x_{n-1}^{-1}\dotsc x_{1}^{-1}$, with $n<m$, such that 
\[
	a_{(1)}a_{i (2)}\dotsb a_{i (m)}x_{1}x_{2}\dotsb x_{n}=1.
\]
According to Dehn's algorithm, either $a_{i (m)}=x_{1}^{-1}$, or there
must exist a block of between $2g+1$ and $4g$ consecutive letters that
can be shortened by using (a cyclic rewriting of) the fundamental
relation \eqref{eq:surfaceRelations}. Since $\beta $ is geodesic, this
block must include the last letter of $\alpha$; and because $a$'s and
$b$'s alternate in the fundamental relation, it must actually begin
with the last letter $a_{i (m)}$, and therefore include at least the
first $2g$ letters $x_{1},\dotsc ,x_{2g}$. Hence, the Dehn shortening
results in
\[
	a_{(1)}a_{i (2)}\dotsb a_{i (m-1)}y_{1}y_{2}\dotsb y_{k}=1,
\]
where $k<m-1$. Therefore, by induction on $m$, there exists $r\geq
m-n\geq 1$ such that 
\[
	a_{(1)}a_{i (2)}\dotsb a_{i (r)}=1.
\]
But this is impossible,  by Dehn. This proves that every self-avoiding
path in $\mathcal{F}_{A}^{+}$ beginning at the root $1$ is geodesic,
and it follows by homogeneity that every self-avoiding path in
$\mathcal{F}_{A}^{+}$ is also geodesic. Finally, this implies that the
image of $\mathcal{F}_{A}^{+}$ in the Cayley graph is a tree.
\end{proof}

\begin{note}\label{note:subtrees}
The presence of large free  semigroups is a general property of
word-hyperbolic groups (see for example \cite{gromov}, Th.~5.3.E), and
for this reason it may well be possible to generalize the arguments
below. The primary obstacle to generalization seems to be in obtaining
suitable \emph{a priori} estimates on the first-passage generating
functions to use in conjunction with Lemma~\ref{lemma:treeCrossing} below.
\end{note}

\subsection{Crossing a tree}\label{ssec:crossingTree}

Say that a path $\gamma $ in the Cayley graph $G^{\Gamma}$
\emph{crosses} a rooted subtree $T$ of degree $d$ if either it visits
the root of the tree, at which time it terminates, or if it crosses
each of the $d$ subtrees $T_{i}$ of $T$ attached to the root. (In the
latter case, the path must terminate at the root of the last subtree
it crosses.) Observe that if $G^{\Gamma}$ is planar, as when $\Gamma
=\Gamma_{g}$ is a surface group, this definition of crossing accords
with the usual topological notion of a crossing. For a vertex $x\in
\Gamma$, let $\mathcal{P} (x;T)$ be the set of all paths starting at
$x$ that cross $T$. Let $\mathcal{P}^{m} (x;T)$ be the set of all
paths in $\mathcal{P} (x;T)$ of length $\leq m$. Define
\begin{align*}
	H_{r} (x;T)&:=\sum_{\gamma \in \mathcal{P} (x;T)}w_{r} (\gamma)
	\quad \text{and}\\
	H^{m}_{r} (x;T)&:=\sum_{\gamma \in \mathcal{P}^{m} (x;T)}w_{r} (\gamma).
\end{align*}
(Recall that $w_{r} (\gamma)$ is the $r-$weight of the path $\gamma$,
defined by \eqref{eq:weight}). The following result is the essence
of the  walkabout argument.

\begin{lemma}\label{lemma:treeCrossing}
Suppose that $F_{r} (1,x)\leq \beta $ for every $x\not =1$ and some
constant $\beta <1$. If $\beta$ is sufficiently small, then for every
rooted subtree $T\subset G^{\Gamma}$ of degree $d\geq 2$ and every
vertex $x\not \in T$,
\begin{equation}\label{eq:crossingIneq}
	H_{r} (x;T)\leq 2\beta .
\end{equation}
\end{lemma}

\begin{proof}
It suffices to show that the inequality holds with $H_{r} (x;T)$
replaced by $H^{m}_{r} (x;T)$ for any $m\geq 1$.  Now the generating
function $H^{m} (x;T)$ is a sum over paths of length $\leq m$. In
order that such a path $\gamma $ starting at $x$ crosses $T$, it must
either visit the root of $T$, or it must cross each of the $d$
offshoot tree $T_{i}$. Since these are pairwise disjoint, a path
$\gamma $ that crosses every $T_{i}$ can be decomposed as $\gamma
=\gamma_{1}\gamma_{2}\dotsb \gamma_{d}$, where $\gamma_{1}$ starts at
$x$ and crosses $T_{1}$, and $\gamma_{i+1}$ starts at the endpoint of
$\gamma_{i}$, in $T_{i}$, and crosses $T_{i+1}$. Each $\gamma_{i}$
must have length at least $1$; hence, since their concatenation has
length $\leq m$, each $\gamma_{i}$ must have length $\leq m-d+1\leq
m-1$.  Since the sum of $w_{r} (\gamma)$ over all paths $\gamma$ from
$x$ to the root of $T$ is no larger than $\beta$, by hypothesis, it
follows that
\[
	\sup_{T}\sup_{x\not \in T}H^{m}_{r} (x;T)
	\leq  \beta +
	\left(\sup_{T}\sup_{x\not \in T}H^{m-1}_{r} (x;T)\right)^{d}
\]
Therefore, since $H^{1}_{r} (x;T)\leq \beta$, for every $m\geq 1$ the
value of $H^{m}_{r} (x;T)$ is bounded above  by the smallest positive
root of the equation 
\[
	y=\beta +y^{d}.
\]
For $\beta >0$ sufficiently small, this root is less than $2\beta$,
regardless of $m$.
\end{proof}

\subsection{Exponential decay of the Green's
function}\label{ssec:expDecay} Assume now that $\Gamma $ has a planar
Cayley graph $G^{\Gamma}$, and that this is embedded
quasi-isometrically in the hyperbolic plane.

\begin{definition}\label{definition:barrier}
Let $\gamma= x_{0}x_{1}\dotsb x_{m}$ be a geodesic segment in the
Cayley graph $G^{\Gamma}$. Say that a vertex $x_{k}$ on $\gamma$ is a
\emph{barrier point} if there are disjoint rooted subtrees
$T_{k},T_{k}'$ in the Cayley graph, both of of outdegree $d\geq 2$,
and neither intersecting $\gamma$, whose roots $y_{k}$ and $z_{k}$ are
vertices neighboring $x_{k}$ on opposite sides of $\gamma$.  Call
$T_{k}\cup T_{k}'\cup \{x_{k} \}$ a \emph{barrier}, and the common
outdegree $d$ the \emph{order} of the barrier.
\end{definition}

Note that a barrier must disconnect the hyperbolic plane in such a way
that the initial and final segments $x_{0}x_{1}\dotsb x_{k-1}$ and
$x_{k+1}x_{k+2}\dotsb x_{m}$ of $\gamma$ lie in opposite
components. In the arguments of \cite{anderson-schoen} and
\cite{ancona:annals}, the region of hyperbolic space separating two
\emph{cones} plays the role of a barrier.

\begin{proposition}\label{proposition:barriers}
Let $\Gamma =\Gamma_{g}$ be the surface group of genus $g\geq
2$. There is a constant $\kappa =\kappa_{g}<\infty$ such that  along every
geodesic segment $\gamma$ of length $\geq \kappa n$ there are $n$
disjoint barriers $B_{i}$, each of order $g$.
\end{proposition}

\begin{note}\label{note:kappa}
The value of the constant $\kappa $ is not
important in the arguments to follow. The argument below shows that
$\kappa_{g}=8g$ will work.
\end{note}

The proof of Proposition~\ref{proposition:barriers} is deferred to
section \ref{ssec:barriers} below.  Given the existence of barriers,
the tree-crossing Lemma~\ref{lemma:treeCrossing}, and the \emph{a
priori} estimate on the Green's function provided by
Proposition~\ref{proposition:vanishingF}, the exponential decay of the
Green's function at the spectral radius follows routinely:

\begin{theorem}\label{theorem:expDecay}
If the genus $g$ is sufficiently large, then the Green's function
$G_{R} (x,y)$ of simple random walk on the surface group $\Gamma_{g}$,
evaluated at the spectral radius $R=R_{g}$, decays exponentially in
the distance $d (x,y)$, that is, there exist constants
$C=C_{g}<\infty$ and $\varrho =\varrho_{g}<1$ such that for every
$x\in \Gamma_{g}$,
\begin{equation}\label{eq:expDecay2}
	G_{R} (1,x)\leq C\varrho^{|x|}.
\end{equation}
\end{theorem}

\begin{proof}
By Proposition \ref{proposition:vanishingF}, for any $\beta >0$ there
exists $g_{\beta}<\infty$ so large that if $g\geq g_{\beta}$ then the
first-passage generating functions of the simple random walk on
$\Gamma_{g}$ satisfy $F_{R} (1,x)<\beta$ for all $x\not =1$. By
Lemma~\ref{lemma:treeCrossing}, the tree-crossing generating functions
$H_{R} (x;T)$ for trees of outdegree $g$ satisfy $H_{R} (x;T)\leq
\beta +\beta^{g}/ (1-\beta^{g})$. If $\beta $ is sufficiently
small then it follows that for any barrier $B=T\cup T'\cup \{y \}$ of
order $g$ and any vertex $x\not \in B$,
\[
	H_{R} (x;T)+H_{R} (x;T')+F_{R} (x,y)<1/2.
\]
By Proposition \ref{proposition:barriers},  a path
$\gamma$ from $1$ to $x$ must cross $|x|/\kappa_{g}$ distinct
barriers. Therefore,
\[
	F_{R} (1,x)\leq 2^{-|x|/\kappa_{g}}.
\]
\end{proof}

\subsection{Action of $\Gamma_{g}$ on the hyperbolic
plane}\label{ssec:ationHyp} 

The surface group $\Gamma =\Gamma_{g}$ acts by hyperbolic isometries
of the hyperbolic plane $\zz{H}$. This action provides a useful
description of the Cayley graph $G^{\Gamma}$, using the tessellation
$\mathcal{T}\:=\{x\mathcal{P} \}_{x\in \Gamma }$ of the hyperbolic
plane $\zz{H}$ by fundamental polygons (``tiles'') $x\mathcal{P}$
(see, e.g., \cite{katok}, chs. 3--4); for the surface group $\Gamma_{g}$ the
polygon $\mathcal{P}$ can be chosen to be a regular $4g-$sided polygon
(cf. \cite{katok}, sec.~4.3, Ex. C).  The tiles serve as the vertices
of the Cayley graph; two tiles are adjacent if they share a
side. Thus, each group generator $a_{i}^{\pm},b_{i}^{\pm}$ maps
$\mathcal{P}$ onto one of the $4g$ tiles that share sides with
$\mathcal{P}$.  The sides of $\mathcal{P}$ (more precisely, the
geodesics gotten by extending the sides) can be labeled clockwise, in
sequence, as 
\[
A_{1},B_{1},\bar{{A}}_{1},\bar{B}_{1},\dotsc,
\bar{B}_{g}
\]
in such a way that each generator $a_{i}$ maps the
\emph{exterior} of the geodesic $A_{i}$ onto the \emph{interior} of
$\bar{A}_{i}$, and similarly $b_{i}$ maps the {exterior} of the
geodesic $B_{i}$ onto the {interior} of $\bar{B}_{i}$. Observe that
$4g$ tiles meet at every vertex of $\mathcal{P}$; for each such
vertex, the successive group elements in some cyclic rewriting of the
fundamental relation \eqref{eq:surfaceRelations}, e.g.,
\[
	a_{1}, a_{1}b_{1}, a_{1}b_{1}a_{1}^{-1}\dotsb ,
\]
map the polygon $\mathcal{P}$ in sequence to the tiles arranged around
the vertex. Also, the full tessellation is obtained by drawing all of
the geodesics $xA_{i},xB_{i},x\bar{A}_{i},x\bar{B_{i}}$, where $x\in
\Gamma$: these partition $\zz{H}$ into the congruent polygons
$x\mathcal{P}$. Call the geodesics $A_{1},B_{1},\dotsc ,\bar{B_{g}}$
\emph{bounding geodesics} of $\mathcal{P}$, and their images by
isometries $x\in \Gamma$ \emph{bounding geodesics} of the
tessellation.

That the semigroups $\mathcal{F}_{A}^{\pm 1},\mathcal{F}_{B}^{\pm1}$
defined in section~\ref{ssec:embeddedTrees} are free corresponds
geometrically to the following important property of the tessellation
$\mathcal{T}$: The exteriors of two bounding geodesics of
$\mathcal{P}$ do not intersect unless the corresponding symbols are
adjacent in the fundamental relation \eqref{eq:surfaceRelations},
e.g., the exteriors of $A_{1}$ and $B_{1}$ intersect, but the
exteriors of $A_{1}$ and $B_{2}$ do not.

\begin{lemma}\label{lemma:geodesics}
Let $\gamma$ be a geodesic segment in the Cayley graph that begins at
$\mathcal{P}$ and on its first step jumps from $\mathcal{P}$ to the
tile $a_{i}\mathcal{P}$ (respectively, $b_{i}\mathcal{P}$, or
$a_{i}^{-1}\mathcal{P}$, or $b_{i}^{-1}\mathcal{P}$). Then $\gamma$
must remain in the halfplane exterior to $A_{i}$ (respectively,
$B_{i},\bar{A}_{i}$, or $\bar{B}_{i}$) on all subsequent jumps.
\end{lemma}

\begin{proof}
By induction on the length of $\gamma$. First,  $\gamma$
cannot recross the geodesic line $A_{i}$ in $\zz{H}$ in
its first $2g+1$ steps, because to do so would require that $\gamma$
cycle through at least $2g+1$ tiles that meet at one of the vertices
of $\mathcal{P}$ on $A_{i}$. This would entail completing more
than half of a cyclic rewriting of the fundamental relation
\eqref{eq:surfaceRelations}, and so $\gamma$ would not be a geodesic
segment in the Cayley graph.

Now suppose that $|\gamma|\geq 2g+1$. Since $\gamma$ cannot complete
more than $2g$ steps of a fundamental relation, it must on some step
$j\leq 2g$ jump to a tile that does not meet $\mathcal{P}$ at a
vertex.  This tile must be on the other side of a bounding geodesic
$C$ that does not intersect $A_{i}$ (by the observation preceding the
lemma). The induction hypothesis implies that $\gamma$ must remain
thereafter in the halfplane exterior to this bounding geodesic, and
therefore in the halfplane exterior to $A_{i}$.
\end{proof}

\subsection{Existence of barriers: Proof of Proposition
\ref{proposition:barriers}}\label{ssec:barriers}  

Let $\gamma$ be a geodesic segment in the Cayley graph. Since $\gamma$
cannot make more than $2g$ consecutive steps in a relator sequence (a
cyclic rewriting of the fundamental relation), at least once  in every
$4g$ steps it must jump across a  bounding geodesic $xL$ into a tile
$x\mathcal{P}$, and  then on the next step jump across a bounding
geodesic $xL'$ that does not meet $xL$. By
Lemma~\ref{lemma:geodesics}, $\gamma$ must remain in the halfplane
exterior to $xL'$ afterwords. Similarly, by time-reversal, $\gamma$
must stay in the halfplane exterior to $xL$ up to the time it enters
$x\mathcal{P}$. Thus, the tile $x\mathcal{P}$ segments $\gamma$ into
two parts, past and future, that live in nonoverlapping halfplanes.

\begin{definition}\label{definition:cutPoint}
If a geodesic segment  $\gamma $ in the Cayley graph $G^{\Gamma}$
enters a tile $x\mathcal{P}$ by crossing a bounding geodesic $xL$ and
exits  by crossing a bounding geodesic $xL'$ that does not intersect
$xL$, then the tile $x\mathcal{P}$ --- or the vertex $x\in \Gamma$ --
is called  a \emph{cut point} for $\gamma$.
\end{definition}

\begin{lemma}\label{lemma:cutPoint}
Let $\gamma$ be a geodesic segment in $G^{\Gamma}$ from $u$ to $v$.
If $x$ is a cut point for $\gamma$, then it is also a barrier
point. Moreover, \emph{every} geodesic segment from $u$ to $v$ passes
through $x$.
\end{lemma}

\begin{proof}
Since $\gamma$ jumps into, and then out of $x\mathcal{P}$ across
bounding geodesics $xL$ and $xL'$ that do not meet, the sides $xL''$
and $xL'''$ of $x\mathcal{P}$ adjacent to the side $xL'$ are distinct
from $xL$. Denote by $y\mathcal{P}$ and $z\mathcal{P}$ the tiles
adjacent to $x\mathcal{P}$ across these bounding geodesic lines $xL''$
and $xL'''$.  For each of these tiles $\tau $, at least one of the
four trees rooted at $\tau $ obtained by translation of the four
semigroups of Proposition~\ref{proposition:embeddedTrees} will lie
entirely in the intersections of the halfplanes \emph{interior} to
$xL$ and $xL'$, by Lemma~\ref{lemma:geodesics}. Therefore, each of the
tiles $y\mathcal{P}$ and $z\mathcal{P}$ is the root of a tree that
does not intersect $\gamma$. These trees, by construction, lie on
opposite sides of $\gamma$. This proves that $x$ is a barrier point.

Suppose now that $\gamma '$ is another  geodesic segment from $u$ to
$v$. If $\gamma '$ did not pass through the tile $x\mathcal{P}$, then
it would have to circumvent it by passing through one of the tiles
$y\mathcal{P}$ or $z\mathcal{P}$. To do this would require either that
it complete a relation or pass through $g$ trees. In either case, the
path $\gamma '$ could be shortened by going through $x\mathcal{P}$.
\end{proof}

To complete the proof of Proposition~\ref{proposition:barriers} it
remains to show that the successive barriers along $\gamma$
constructed above are pairwise disjoint. But the attached trees at the
tiles $y\mathcal{P}$ and $z\mathcal{P}$ were chosen in such a way that
each lies entirely in the intersections of the halfplanes
\emph{interior} to $xL$ and $xL'$. The past and future segments of
$\gamma$ lie in the \emph{exteriors}. Hence, at each new barrier along
(say) the future segment, the attached trees will lie in halfplanes
contained in these exteriors, and so will not intersect the barrier at
$x\mathcal{P}$.

\qed

\subsection{Ancona's inequality}\label{ssec:anconaInequality} The
Ancona inequalities \eqref{eq:ancona} state that the major
contribution to the Green's function $G_{R} (x_{0},x_{m})$ comes from
random walk paths that pass within a bounded distance of $x_{n}$. To
prove this it suffices, by Lemma~\ref{lemma:cutPoint}, to
show that \eqref{eq:anconaF} holds for \emph{cut points}
$x_{m}$. The key to this is that a path from $x_{0}$ to $x_{m}$ that
does \emph{not} pass within distance $n$ of $x_{m}$ must cross
$g^{n-1}$ trees of outdegree $g$.

\begin{lemma}\label{lemma:walkabout}
Let $\gamma$ be a geodesic segment from $u$ to $v$ that passes through
the root vertex $1$, and suppose that vertex $1$ is a cut point
for $\gamma$. Assume that both $u,v$ are exterior to the sphere
\begin{equation}\label{eq:sphere}
	S_{n}:=\{x\in \Gamma \, : \, |x|=n\}
\end{equation}
of radius $n$ in the Cayley graph $G^{\Gamma}$ centered
at $1$.  If $F_{R}(1,x)\leq \beta$ for all vertices $x\not =1$, then
\begin{equation}\label{eq:walkaboutBound}
	G_{R} (u,v;G^{\Gamma}\setminus S_{n})
	\leq 2 (2\beta )^{g^{n-1}}.
\end{equation}
\end{lemma}

\begin{proof}
Since both $u,v$ are exterior to $S_{n}$, the restricted Green's
function is the sum over all paths from $u$ to $v$ that do not enter
the sphere $S_{n}$ (recall definition \eqref{eq:fpgf}). Since $1$ is a
barrier point for $\gamma$, there are trees $T,T'$ of outdegree $g$
with roots adjacent to $1$ on either side of $\gamma$. A path from $u$
to $v$ that does not enter $S_{n}$ must cross either $T$ or $T'$, and
it must do so without passing within distance $n-1$ of the root. Thus,
it must cross $g^{n-1}$ disjoint subtrees of either $T$ or
$T'$. Consequently, the result follows from
Lemma~\ref{lemma:walkabout}.
\end{proof}

\begin{proposition}\label{proposition:ancona}
For all sufficiently large $g$, there exists $C=C_{g}<\infty$ such
that the 
Green's function of the simple random walk on the surface group
$\Gamma_{g}$ satisfy the Ancona inequalities: In particular, for every
geodesic segment $x_{0}x_{1}x_{2}\dotsb x_{m}$, every $1<n<m$, and
every $1\leq r\leq R$,
\begin{equation}\label{eq:anconaF}
	G_{r} (x_{0},x_{m})\leq C G_{r} (x_{0},x_{n})G_{r} (x_{n},x_{m}).
\end{equation}
\end{proposition}

\begin{proof}
It is certainly true that for each distance $m<\infty$ there is a
constant $\infty >C_{m}\geq 1$ so that \eqref{eq:anconaF} holds for
all geodesic segments of length $m$, because (by homogeneity of the
Cayley graph) there are only finitely many possibilities. The problem
is to show that the constants $C_{m}$ remain bounded as $m \rightarrow
\infty$.

As noted above, it suffices to consider only \emph{cut points}
$x_{n}$ along the geodesic segment $\gamma$. For ease of notation,
assume that $\gamma $ has been translated so that the cut point
$x_{n}=1$ is the root vertex of the Cayley graph, and write
$u=x_{0}$ and $v=x_{m}$ for the initial and terminal points. Assume
also that $d (u,1)\leq m/2$; this can be arranged by switching the
endpoints $u,v$, if necessary. Thus, there is a cut point $w$ on the
geodesic segment between $1$ and $v$ so that $.7 m\leq d (u,w)\leq .8$.
Let $S=S_{\sqrt{m}} (w)$ and $B=B_{\sqrt{m}} (w)$ be the sphere  and
ball, respectively, of common radius $\sqrt{m}$, centered at $w$.
Any path from $u$ to $v$ (or any path from $1$ to $v$) must either
pass through the ball $B$ or not; hence
\[
	G_{r} (u,v)=G_{r} (u,v;B^{c})+\sum_{z \in S} G_{r} (u,z)G_{r} (z,v;B^{c}).
\]
If $m$ is sufficiently large that $\sqrt{m}<.1 m$, then any point
$z\in S$ must be at distance
\[
	.6m \leq d (u,z)\leq .9 m
\]
from $u$. Moreover, by Lemma~\ref{lemma:cutPoint}, every geodesic segment from
$u$ to $z$  must pass through $1$. (This follows because $1$ is a cut point
for $\gamma$.) Similarly, since $w$ is also a cut point, every geodesic
segment from $1$ to $v$ passes through $w$. Consequently, for every
$z\in S$,
\[
	G_{r} (u,z)\leq C_{[.9m]} G_{r} (u,1)G_{r} (1,z).
\]
By Lemma \ref{lemma:walkabout}, 
\[
	G_{r} (u,v;B^{c})\leq G_{R} (u,v;B^{c})\leq  2\alpha^{g^{\sqrt{m}}} 
\]
where $\alpha =2\beta <1/2$, provided the
genus $g$ is sufficiently large.  On the other hand, the Harnack
inequalities ensure that for some $\varrho >0$ and all $r\geq 1$
\begin{align*}
	G_{r} (u,1)&\geq \varrho^{m} \quad \text{and}\\
	G_{r} (1,v)&\geq \varrho^{m} 
\end{align*}
Therefore,
\begin{align*}
	G_{r} (u,v)&=G_{r} (u,v;B^{c})+\sum_{z \in S} G_{r} (u,z)G_{r} (z,v;B^{c})\\
	      &\leq 2\alpha^{g^{\sqrt{m}}} 
	       +C_{[.9m]}\sum_{z \in S}G_{r} (u,1) G_{r} (1,z)G_{r} (z,v;B^{c})\\
	       &\leq 2\alpha^{g^{\sqrt{m}}}  +
	       C_{[.9m]} G_{r} (u,1) G_{r} (1,v)\\
	       &\leq (1+2\alpha^{g^{\sqrt{m}}}/\varrho^{2m})
	        C_{[.9m]} G_{r} (u,1) G_{r} (1,v).
\end{align*}
This shows that 
\[
	C_{m}\leq (1+2\alpha^{g^{\sqrt{m}}}/\varrho^{2m})
	        C_{[.9m]} ,
\]
and it now follows routinely that the constants $C_{m}$ remain bounded
as $m \rightarrow \infty$.
\end{proof}

\section{Automatic structure}\label{sec:cannon}

\subsection{Strongly Markov groups and
hyperbolicity}\label{ssec:strongMarkov}

A finitely generated group $\Gamma $ is said to be \emph{strongly Markov}
(fortement Markov -- see \cite{ghys-deLaHarpe}) if for each finite,
symmetric generating set $A$ there exists a finite directed graph 
$\mathcal{A}= (V,E,s_{*})$ with distinguished vertex $s_{*}$ (``start'') and a
labeling $\alpha :E \rightarrow A$ of edges by generators that meets
the following specifications. Let 
\begin{equation}\label{eq:automatonPaths}
	\mathcal{P}:=\{\text{finite paths  in $\mathcal{A}$ starting at $s_{*}$} \},
\end{equation}
and for each path $\gamma =e_{1}e_{2}\dotsb e_{m}\in \mathcal{P}$, denote by 
\begin{align}\label{eq:inducedPath}
\alpha (\gamma)&= \text{path in $G^{\Gamma }$ through $1,\alpha
	(e_{1}),\alpha (e_{1})\alpha (e_{2}),\dotsc$}, \quad \text{and}\\
\notag 
	\alpha_{*} (\gamma)&=\text{right endpoint of } \;\alpha (\gamma)
		   =\alpha (e_{1})\alpha (e_{2}) \dotsb \alpha(e_{m}).
\end{align}

\begin{definition}\label{definition:automaticStructure}

The labeled automaton $(\mathcal{A},\alpha )$ is a strongly Markov
automatic structure for $\Gamma$ if:

\begin{enumerate}
\item [(A)] No edge $e\in E$ ends at $s_{*}$.
\item [(B)] Every vertex $v\in V$ is
accesssible from the start state $s_{*}$.
\item [(C)]  For every path $\gamma \in
\mathcal{P}$, the path $\alpha (\gamma)$ is a geodesic path in $G^{\Gamma}$.
\item [(D)]  The  endpoint mapping $\alpha_{*}:\mathcal{P}\rightarrow
\Gamma $ induced by $\alpha$ is a bijection of
$\mathcal{P}$ onto $\Gamma$.
\end{enumerate}
\end{definition}

 \begin{theorem}\label{theorem:cannon}
Every word hyperbolic group is strongly Markov.
\end{theorem}

See \cite{ghys-deLaHarpe}, Ch.~9, Th.~13. The result is essentially
due to Cannon (at least in a more restricted form) --- see
\cite{cannon:1},  \cite{cannon:3} --- and in important
special cases (cocompact Fuchsian groups) to Series
\cite{series}. Henceforth, we will call the directed graph
$\mathcal{A}= (V,E,s_{*})$ the \emph{Cannon automaton} (despite the
fact that it is not quite the same automaton as constructed in
\cite{cannon:1}).

\subsection{Automatic structures for the surface
groups}\label{ssec:autoSurface}

The existence of an automatic structure will be used to connect the
behavior of the Green's function at infinity to the theory of Gibbs
states and Ruelle operators (see \cite{bowen}, ch.~1). For these
arguments, it is not important that the group $\Gamma$ be a surface
group; only the conclusions of
Theorem~\ref{theorem:holderMartinKernel} and
Corollary~\ref{corollary:bi} will be needed. Nevertheless, we note
here that an automatic structure $\mathcal{A}$ for the surface group
$\Gamma_{g}$ is easily constructed. Let $A=A_{g}=\{a_{i}^{\pm},
b_{i}^{\pm}\}$ be the standard generating set, with the generators
satisfying the basic relation \eqref{eq:surfaceRelations}. Define the
set $V$ of vertices for the automaton to be the set of all reduced
words in the generators of length $\leq 2g$, with $s_{*}=$ the empty
word. Directed edges are set according to the following rules:
\begin{itemize}
\item [(A)] If a (reduced) word $w'$ is obtained by adding a single
letter $x$ to the end of word $w$, then draw an edge $e (w,w')$ from
$w$ to $w'$, and label it with the letter $x$.
\item [(B)] If a word $w'$ of maximal length $2g$ is
obtained from another word $w$ of length $2g$ by deleting the first
letter and adding a new letter $x$ to the end, then draw an edge $e
(w,w')$ from $w$ to $w'$ with label $x$ \emph{unless} the word $wx$
constitutes the first $2g+1$ letters of a cyclic permutation of the
basic relation \eqref{eq:surfaceRelations}. 
\end{itemize}
That properties (C)--(D) of
Definition~\ref{definition:automaticStructure} are satisfied follows
from Dehn's algorithm. The words of maximal length $2g$ are the
\emph{recurrent vertices} of this automaton, while the words of length
$<2g$ are the \emph{transient vertices} (see
sec.~\ref{ssec:restrictions} below for the definitions). It is
easily verified that for any vertex $w$ and any \emph{recurrent} vertex
$w'$, there is a path in the automaton from $w$ to $w'$.

\subsection{Recurrent and transient vertices}\label{ssec:restrictions} 

Let $\mathcal{A}$ be a Cannon automaton for the group $\Gamma$ with
vertex set $V$ and (directed) edge set $E$.  Call a vertex $v\in V$
\emph{recurrent} if there is a path in $\mathcal{A}$ of length $\geq
1$ that begins and ends at $v$; otherwise, call it \emph{transient}.
Call a vertex $v$ \emph{terminal} if there is no directed edge leading
out of $v$. Denote by $\mathcal{A}_{R}$ the restriction
of the digraph $\mathcal{A}$ to the set
$\mathcal{R}$ of recurrent vertices. For certain hyperbolic groups ---
among them the surface groups --- the automatic structure can be
chosen so that the digraph $\mathcal{A}_{R}$ is \emph{connected} (see
\cite{series}) and has no terminal vertices.  Henceforth we restrict
attention to word-hyperbolic groups with this property:

\begin{assumption}\label{assumption:irreducibility}
The automatic structure can be chosen so that the digraph
$\mathcal{A}_{R}$ is connected, and so that there are no terminal vertices.
\end{assumption}

Assumption \ref{assumption:irreducibility} implies that every path
$\gamma \in \mathcal{P}$ beginning at the vertex $s_{*}$ has the form
\begin{equation}\label{eq:transient-recurrent}
\gamma =\tau \varrho
\end{equation}
where $\tau$ is a path of length $\geq 1$ in the set of transient
vertices whose last edge connects a transient vertex to a recurrent
vertex $v_{\tau }$, and $\varrho$ is a path in the set of recurrent
vertices that starts at $v$.  There are only finitely many possible
transient prefixes $\tau$, since no transient vertex can be visited
twice by a path $\gamma$.

\begin{assumption}\label{assumption:mixing}
The incidence matrix of the digraph $\mathcal{A}_{R}$ is aperiodic.
\end{assumption}

Both assumptions hold for all surface
groups. Assumption~\ref{assumption:mixing} is for ease of exposition
only --- the results and arguments below can be modified to account
for any periodicities that might arise if the assumption were to
fail. Assumption~\ref{assumption:irreducibility}, however, is
essentially important.  Given
Assumptions~\ref{assumption:irreducibility}--\ref{assumption:mixing} ,
a further simplification of the automaton is possible:

\begin{lemma}\label{lemma:tp}
The automaton $\mathcal{A}$ can be modified so that all transient
prefixes $\tau$ are of the same length $K$. Moreover, the modification
can be made in such a way that for every  path $\varrho$ in the set of
recurrent vertices there is at least one transient prefix $\tau$ such
that the concatenation $\tau \varrho$ is a path in the automaton
$\mathcal{A}$. 
\end{lemma}

\begin{proof}
The recurrent vertices of the modified automaton will be the same as
in the original, as will the edges among them; only the set of
transient vertices will be modified. Thus, the digraph
$\mathcal{A}_{R}$ will not be changed.  Let $K$ be the length of the
longest transient prefix.  Replace the set $T$ of transient vertices
by the set $T^{*}$ of \emph{paths} of length $J\leq K$ in the
automaton that start at
transient vertices. Note that any such path of length $K$ must end in
a recurrent vertex. For any two paths $\gamma ,\gamma '\in T^{*}$, draw
an edge from $\gamma$ to $\gamma '$ if $\gamma '$ is obtained by
adding a single vertex to the end of $\gamma$.  For any path $\gamma
\in T^{*}$ of length $K$ and any recurrent vertex $v$, draw an edge from
$\gamma$ to $v$ if there is an edge in the original automaton
$\mathcal{A}$ from the recurrent vertex $v'$ at the end of $\gamma$ to
$v$.  Finally, construct the edge-labeling $\alpha' :E' \rightarrow A$
by projection: for instance, if
\begin{align*}
	\gamma &=v_{1}v_{2}\dotsb v_{J} \quad \text{and}\\
	\gamma '&=v_{1}v_{2}\dotsb v_{J}v_{J+1}
\end{align*}
then label the edge from $\gamma$ to $\gamma '$ the same way that the
edge from $v_{J}$ to $v_{J+1}$ was labelled in the original automaton.

A similar argument can be made to prove the second assertion. By
Assumption~\ref{assumption:mixing}, there exists an integer $L\geq 1$
so that for any two recurrent vertices $v,w$ there is a path of length
$L$ from $v$ to $w$. Replace the transient vertices of the original
automaton by paths of length $K+L$ that start at $s_{*}$, and modify
the edges as in the preceding paragraph.
\end{proof}

\subsection{Symbolic dynamics}\label{ssec:symbolicDynamics} We shall
assume for the remainder of the paper that the automaton $\mathcal{A}$
has been chosen so that all transient prefixes have the same length
$K$, and so that for every path $\varrho$ in the set of recurrent
vertices there is at least one transient prefix $\tau$ such that the
concatenation $\tau \varrho$ is a path in the automaton. Let
$\mathcal{R}$ be the set of recurrent vertices of $\mathcal{A}$, and
for each transient prefix $\tau$ denote by $v_{\tau}$ the terminal
vertex of $\tau$ (which is necessarily recurrent).
 Set
\begin{align*}
	\Sigma &=
	 \{\text{semi-infinite  paths in} \;\mathcal{R} \},\\ 
       \Sigma_{\tau}&= \{\text{semi-infinite  paths in} \;\mathcal{R}\;
       \text{that begin at} \; v_{\tau} \}, \\
       \tilde{\Sigma }&=\{\text{bi-infinite  paths in} \;\mathcal{R} \},\\ 
	\Sigma^{n}&=\{\text{paths of length $n$ in}\; \mathcal{R}\},\\
	\Sigma^{n}_{\tau}&=\{\text{paths of length $n$ in}\;
	\mathcal{R} \; \text{that begin at} \; v_{\tau}\}, \quad \text{and}\\
	\Sigma^{*}&=\cup_{n=0}^{\infty}\Sigma^{n},
\end{align*}	
and let $\sigma $ be the forward shift operator. The spaces $\Sigma^{*}\cup
\Sigma$ and $\tilde{\Sigma}$ are given metrics in the usual way, that is,
\begin{equation}\label{eq:sigma-metric}
	d (\omega ,\omega ')=2^{-n (\omega ,\omega ')}
\end{equation}
where $n (\omega ,\omega ')$ is  the maximum integer $n$
such that $\omega_{j}=\omega '_{j}$ for all $|j|<n$). With the
topology induced by $d$ the space $\Sigma$ is a Cantor set, and
$\Sigma $ is the set of accumulation points of $\Sigma^{*}$.  Observe
that, relative to the metrics $d$, H\"{o}lder-continuous real-valued
functions on $\Sigma$ pull back to H\"{o}lder-continuous functions on
$\tilde{\Sigma}$.

The sets $\Sigma_{\tau}$ need not be pairwise disjoint, but their
union is $\Sigma$. For each $\omega \in \Sigma_{\tau}$  the
concatenation $\tau \omega$ is an infinite path in $\mathcal{A}$
beginning at $s_{*}$, and hence  projects via the edge-labeling map
$\alpha $ to a geodesic ray in $G^{\Gamma}$ starting at the vertex $1$
(more precisely, the sequence of finite prefixes of $\tau \omega$
project to the vertices along a geodesic ray). Each geodesic ray in
$G^{\Gamma}$ must converge in the Gromov topology to a point of
$\partial \Gamma$, so $\alpha$ induces on each $\Sigma_{\tau}$ a
mapping to $\partial \Gamma$. By construction, this mapping is
H\"{o}lder continuous relative to any visual metric on $\partial
\Gamma$. Each point $\zeta \in \partial \Gamma$ is the limit of a
geodesic ray corresponding to a semi-infinite path $\tau \omega$ in
$\mathcal{A}$ that begins at $s_{*}$, so $\partial \Gamma$ is the
union of the images of the sets $\Sigma_{\tau}$.

In a somewhat different way, the edge-labeling map $\alpha$ determines
a map from the space $\tilde{\Sigma}$ to the set of two-sided
geodesics in $G^{\Gamma}$ that pass through the vertex $1$. This map
is defined as follows: if $\omega \in \tilde{\Sigma}$ then the image
of $\omega$ is the two-sided geodesic that passes through
\[
	\dotsb , \alpha (\omega_{-1}^{-1})\alpha (\omega_{0})^{-1},
	\alpha (\omega_{0}^{-1}),1,\alpha (\omega_{1}),\alpha
	(\omega_{1}) \alpha (\omega_{2}), \dotsc ,
\]
equivalently, it is the concatenation of the geodesic rays starting at
$1$ that are obtained by reading successive steps from the sequences
\[
	\omega_{1}\omega_{2}\omega_{3}\dotsb 
	\quad \text{and}\quad 
	\omega_{0}^{-1}\omega_{-1}^{-1}\omega_{-2}^{-1}\dotsb ,
\]
respectively. Each of these geodesic rays converges to a point of
$\partial \Gamma$, so $\alpha$ induces a  mapping
from $\tilde{\Sigma}$ into $\partial \Gamma \times \partial
\Gamma$. This mapping is neither injective nor surjective, but it is
H\"{o}lder-continuous. 

Since every transient prefix has length $K$, for each
$m\geq K$ the sphere $S_{m}$ of radius $m$ in $G^{\Gamma}$ is in
one-to-one correspondence with
\begin{equation}\label{eq:sphereDecomp}
	\bigcup_{\tau}\Sigma^{m-K}_{\tau}
\end{equation}
where the union $\cup_{\tau}$ is over all transient prefixes $\tau$.

\begin{corollary}\label{corollary:sphereGrowth}
Let $T$ be the incidence matrix of the digraph $\mathcal{A}_{R}$, and
let $\zeta $ its spectral radius. If
Assumptions~\ref{assumption:irreducibility} and
\ref{assumption:mixing} hold, then $\zeta >1$, and there exists $C>0$
such that
\begin{equation}\label{eq:regularSphereGrowth}
	|S_{m}|\sim C\zeta^{m} \quad \text{as} \;\; m \rightarrow \infty .
\end{equation}
\end{corollary}

\begin{proof}
This follows directly from the Perron-Frobenius theorem, with the
exception of the assertion that the spectral radius $\zeta >1$. That
$\zeta >1$ follows from the fact that the group $\Gamma$ is
nonelementary. Since $\Gamma $ is nonelementary, it is nonamenable,
and so its Cayley graph has positive Cheeger constant; this implies
that $|S_{n}|$ grows exponentially with $n$.
\end{proof}

\begin{corollary}\label{corollary:positiveEntropy}
The shift $(\Sigma, \sigma)$ is
topologically mixing and has positive topological entropy.
\end{corollary}

\begin{proof}
Topological ergodicity follows from
Assumption~\ref{assumption:irreducibility}, and topological mixing
from Assumption~\ref{assumption:mixing}.    That the the subshift
$(\Sigma,\sigma)$ has positive topological entropy follows from the
exponential growth of the group,
cf. Corollary~\ref{corollary:sphereGrowth}.
\end{proof}

Corollary~\ref{corollary:positiveEntropy} ensures that the standing
hypotheses on the shift $(\Sigma ,\sigma)$ of Bowen \cite{bowen},
ch.~1, are satisfied. The machinery of thermodynamic formalism and
Gibbs states developed in \cite{bowen} applies to H\"{o}lder
continuous functions on $\Sigma$ (or on $\Sigma^{*}\cup \Sigma$). To
make use of this machinery, we will lift the Martin kernel from
$\partial \Gamma$ to the sequence space $\Sigma$. For this the results
of Theorem~\ref{theorem:holderMartinKernel} and
Corollary~\ref{corollary:bi} are crucial, as they  ensure that the
lift of the Martin kernel is H\"{o}lder continuous.
The lift is defined as follows. Fix $\omega \in \Sigma$, and let
\[
	\omega^{(n)}=\omega_{1}\dotsb \omega_{n}
\]
be a prefix of $\omega $. There exists at least one transient prefix
$\tau$  such that $\tau \omega^{(n)}$ is a path in the automaton.
(This path must begin at $s_{*}$.) Set
\[
	\varphi_{r} (\omega^{(n)}):=\log \frac{G_{r}
(\alpha_{*} (\tau ),\alpha_{*} (\tau
\omega^{(n)}))}{G_{r} (\alpha_{*} (\tau
\omega_{0}),\alpha_{*} (\tau \omega^{(n)}))};
\]
this ratio is independent of the choice of $\tau $,
since $\alpha_{*} (\tau)^{-1}\alpha_{*} (\tau
\omega_{1})$ is determined solely by the edge
$\omega_{1}$. Theorem~\ref{theorem:holderMartinKernel} implies that
the log ratio converges as $n \rightarrow \infty$, and that the limit
function
\begin{equation}\label{eq:defCocycle}
	\varphi_{r} (\omega ) =\lim_{n \rightarrow \infty}\varphi_{r} (\omega^{(n)} )
		    =\log \frac{K_{r} (\alpha_{*} (\tau ),\alpha_{*}
	(\tau \omega))}{K_{r} (\alpha_{*} (\tau  \omega_{1} ),\alpha_{*}
	(\tau \omega))}
\end{equation}
is H\"{o}lder continuous (relative to an exponent independent of $r$)
in $\omega$.  Furthermore, since the mapping $r\mapsto K_{r}
(x,\zeta)$ is continuous in the H\"{o}lder norm for some exponent
independent of $r$ (Theorem~\ref{theorem:holderMartinKernel}), the
mapping $r\mapsto \varphi_{r}$ is continuous relative to the
H\"{o}lder norm for functions on the sequence space $\Sigma $.  By
construction,
\begin{equation}\label{eq:cocycle}
	 G_{r} (\alpha_{*} (\tau \varrho),\alpha_{*} (\tau
	 \varrho\omega))
	 =\exp \{S_{n}\varphi_{r}(\omega) \}
\end{equation}
where (in Bowen's notation \cite{bowen}) 
\begin{equation*}
	 S_{n}\varphi :=\sum_{j=0}^{n-1}\varphi \circ \sigma^{j}.
\end{equation*}
 (Unfortunately, the notation $S_{n}\varphi$
conflicts with the notation $S_{m}$ for the sphere of radius $m$ in
$\Gamma$; however, both notations are standard, and the meaning should
be clear in the following by context.)

\section{Thermodynamic formalism}\label{sec:thermo} \subsection{Gibbs
states: background}\label{sec:gibbs} According to a fundamental
theorem of ergodic theory (cf. \cite{bowen}, Th.~1.2 and sec.~1.4),
for each H\"{o}lder continuous function $\varphi :\tilde{\Sigma
}\rightarrow \zz{R}$ there is a unique \emph{Gibbs state} $\mu
=\mu_{\varphi }$. A Gibbs state for the potential $\varphi$ is by
definition a shift-invariant probability measure $\mu$ on
$\tilde{\Sigma } $ for which there are constants $0<C_{1}<C_{2}<\infty
$ such that
\begin{equation}\label{eq:gibbs}
	C_{1}\leq \frac{\mu (\tilde{\Sigma }^{n} (\omega ))}{\exp\{S_{n}\varphi
	(\omega)-n \text{\rm Pressure} (\varphi )\}} \leq C_{2} 
\end{equation}
for all $n\geq 1$ and all $\omega  \in \Sigma $, where
\[
	\tilde{\Sigma }^{n} (\omega ):=\{\omega '\in \tilde{\Sigma } \,:\,
	\omega'_{j}=\omega_{j} \;\forall \,j\leq n \}
\]
and $\text{Pressure} (\varphi)$ denotes the topological pressure of
$\varphi$ (cf. \cite{bowen}, ch.~1, sec.~D).
Furthermore, if $\varphi (\omega)$ depends only on the forward
coordinates $\omega_{0}\omega_{1}\dotsb$ of $\omega$ (as is the case
for the functions $\varphi_{r}$ defined by \eqref{eq:defCocycle}) then
the Gibbs state $\mu_{\varphi }$ is related to the
Perron-Frobenius eigenfunction $h_{\varphi }$ and eigenmeasure
$\nu_{\varphi }$ of the Ruelle operator $\mathcal{L}_{\varphi}$
associated with $\varphi $ (cf. \cite{bowen}, ch.~1, sec.~C) by
\begin{equation}\label{eq:gibbsStateDecomp}
	d\mu_{\varphi }= h_{\varphi }d\nu_{\varphi }
\end{equation}
provided $h_{\varphi }$ and $\nu_{\varphi }$ are normalized so that
$\nu_{\varphi}$ and $h_{\varphi}\nu_{\varphi }$ both have total mass
$1$. This implies (by a standard argument in regular perturbation
theory) that the mapping $\varphi \mapsto \mu_{\varphi}$ is continuous
relative to the weak topology on measures and the H\"{o}lder topology
on functions. It follows that for the functions $\varphi_{r}$ defined
in sec.~\ref{ssec:symbolicDynamics} the measures
$\mu_{r}:=\mu_{\varphi_{r}}$ vary continuously with $r\in (0,R]$, and
that the constants $C_{1},C_{2}$ in \eqref{eq:gibbs} can be chosen to
be independent of $r\in [1,R]$. 

\begin{lemma}\label{lemma:gibbsCharacterization}
Let $\varphi :\Sigma \rightarrow \zz{R}$ be H\"{o}lder
continuous and let $\mu =\mu_{\varphi}$ be the Gibbs state on
$\tilde{\Sigma}$ with potential
function $\varphi$. There exists a positive, H\"{o}lder continuous 
function $\psi=\psi_{\varphi} :\tilde{\Sigma }\rightarrow
\zz{R}$ such that for every $\omega \in \tilde{\Sigma }$,
\begin{equation}\label{eq:gibbsCharacterization}
	\lim_{n \rightarrow \infty} \frac{\mu (\tilde{\Sigma}^{n} (\omega
	))\psi (\sigma^{n}\omega )^{-1} }{\exp\{S_{n}\varphi 
	(\omega)-n \text{\rm Pressure} (\varphi )\}} =h_{\varphi} (\omega),
\end{equation}
where $h_{\varphi}$ is the normalized Perron-Frobenius eigenfunction
of the Ruelle operator $\mathcal{L}_{\varphi}$.  Moreover, the
convergence is exponential, that is, there exist constants $C<\infty$
and $0<r<1$ such that for every $n\geq 1$ and every $\omega \in
\tilde{\Sigma }$,
\begin{equation}\label{eq:gibbsCharacterizationRate}
	\Bigl\lvert \frac{\mu (\Sigma^{n} (\omega ))\psi
	(\sigma^{n}\omega )^{-1} }{\exp\{S_{n}\varphi 
	(\omega)-n \text{\rm Pressure} (\varphi )\}}-h_{\varphi } (\omega)
	\Bigr\rvert \leq Cr^{n}. 
\end{equation}
\end{lemma}

The proof is deferred to the end of the section so as not to distract
from the main line of argument.  

 Any Gibbs state $\mu_{\varphi}$ on $\tilde{\Sigma }$ is
ergodic and mixing, so ergodic averages of
$\mu_{\varphi}-$integrable functions on $\tilde{\Sigma}$ converge
almost surely.  In the arguments of sec.~\ref{sec:criticalExp} below
it will be necessary to use this convergence simultaneously for a
continuously parametrized family of Gibbs states (one for every value
of $r\in [1,R]$). The following result asserts that under suitable
hypotheses the convergence in the ergodic theorem is uniform.

\begin{proposition}\label{proposition:uniformErgodicity}
Assume that the mappings $r\mapsto \varphi_{r}$ and $r\mapsto g_{r}$
are continuous relative to the H\"{o}lder norm. Let $\mu_{r}$ be the
Gibbs state with potential function $\varphi_{r}$. Then the
expectations $\int g_{r}\, d\mu_{r}$ vary continuously with $r$, and
for each $\varepsilon >0$ there exist constants $C<\infty$ and
$0<\varrho <1$ such that for all $r\in [1,R]$ and all $n\geq 1$,
\begin{equation}\label{eq:uniformLD}
	\mu_{r}\left\{\omega \in \tilde{\Sigma } \,: \, \Bigl\lvert \frac{S_{n}g_{r}
	(\omega)}{n}-\int g_{r}\,d\mu_{r} \Bigr\rvert \geq
	\varepsilon\right\}\leq C\varrho^{n} .
\end{equation}
\end{proposition}

The proof is given at the end of the section, following that of 
Lemma~\ref{lemma:gibbsCharacterization}. The hypothesis that the
functions $g_{r}$ be H\"{o}lder continuous, and not merely continuous,
is necessary. 

\subsection{Gibbs states and Green's function on
spheres}\label{ssec:greenSpheres}
Denote by $\lambda_{r,m}$ the probability measure on the sphere
$S_{m}\subset \Gamma $ with density proportional to $G_{r} (1,x)^{2}$,
that is, such that 
\begin{equation}\label{eq:lambda-def}
	\lambda_{r,m} (x)=\frac{G_{r} (1,x)^{2}}{\sum_{y\in
	S_{m}}G_{r} (1,y)^{2}} \quad \text{for all} \;\; x\in S_{m}.
\end{equation}
Recall that $S_{m}$ is in one-to-one correspondence with the paths of
length $m$ in the automaton $\mathcal{A}$ that begin at $s_{*}$; in
particular, each $x\in S_{m}$ corresponds uniquely to a path $\tau
\omega $, where $\tau $ is a transient prefix (necessarily of length
$K$) and $\omega\in \Sigma^{m-K} $ is a path in the set of recurrent
vertices of $\mathcal{A}$. Thus, there is a surjection $x\mapsto
\omega $ from $S_{m}$ to $\Sigma^{m-K}$. Denote by
$\lambda^{*}_{r,m}$ the pushforward to $\Sigma^{m-K}$ of
$\lambda_{r,m}$ via this surjection. By \eqref{eq:sphereDecomp}, the
measure $\lambda^{*}_{r,m}$ is related to the Green's function by
\begin{equation}\label{eq:measuresOnSm}
	\lambda^{*}_{r,m} (\varrho )=\frac{\sum_{\tau :\tau \varrho \in
	\mathcal{P}} G_{r} (1,\alpha_{*} (\tau \varrho ))^{2}}{\sum_{y\in
	S_{m}}G_{r} (1,y)^{2}} \quad \text{for all} \;\; \varrho \in
	\Sigma^{m-K},
\end{equation}
where the sum is over all transient prefixes $\tau $ such that the
concatenation $\tau \varrho $ is a path in $\mathcal{A}$ (starting at
$s_{*}$). Hence, by equation~\eqref{eq:cocycle},
\begin{equation}\label{eq:inducedMeasures}
	\lambda^{*}_{r,m} (\varrho )=\frac{\sum_{\tau :\tau \varrho \in
	\mathcal{P}}G_{r} (1,\alpha_{*} (\tau ))^{2}\exp
	\{2S_{m-K}\varphi_{r} (\varrho ) \})}{\sum_{y\in S_{m}}G_{r} (1,y)^{2}} 
\end{equation}

\begin{proposition}\label{proposition:absolutelyCont}
For each $r\in [1,R]$,
the values $\lambda^{*}_{r,m} (\varrho)$ are the cylinder
probabilities of a Borel probability measure $\lambda^{*}_{r}$ on
$\Sigma$, that is, for each $m>K$ and each
$\omega \in \Sigma$, 
\begin{equation}\label{eq:cylinderProbs}
		\lambda^{*}_{r} (\Sigma^{m-K}
		(\omega))=\lambda^{*}_{r,m} (\omega^{(m-K)})
\end{equation}
where $\omega^{(m-K)}$ denotes the path of length $m-K$ consisting of
the first $m-K$ letters of $\omega$.  This probability measure is
absolutely continuous relative to the Gibbs state $\mu_{r}$ with
potential function $2\varphi_{r}$.  Furthermore, there exist constants
$0<C=C (r;2)<\infty$ such that as $m \rightarrow \infty$,
\begin{equation}\label{eq:sphereAsymptotics}
	\sum_{x\in S_{m}} G_{r} (1,x)^{2}
	\sim  C \exp \bigg\{m \text{\rm Pressure} (2 \varphi_{r}) \bigg\}.
\end{equation}
\end{proposition}

\begin{proof}
For any $\omega \in \Sigma^{*}\cup \Sigma$, the set of prefixes $\tau$
such that $\tau \omega$ is a path in the automaton $\mathcal{A}$
depends only on the  first entry of $\omega$, and consequently, so
does  the factor 
\[
	g_{r} (\omega):=\sum_{\tau :\tau \omega \in
	\mathcal{P}}G_{r} (1,\alpha_{*} (\tau ))^{2}.
\]
It follows trivially that $g_{r}$ is a positive, H\"{o}lder continuous
function of $\omega$. Denote by $\mu_{r}$ the Gibbs state with
potential function $\varphi_{r}$, and let $\psi =\psi_{r}$ be as in
Lemma~\ref{lemma:gibbsCharacterization}. Equations
\eqref{eq:inducedMeasures} and \eqref{eq:gibbsCharacterizationRate}
imply that for any $\omega \in \Sigma^{m-K}$,
\begin{equation}\label{eq:measureRatio}
	\frac{\mu_{r} (\Sigma^{m-K} (\omega))}{\lambda^{*}_{r,m}
	(\omega)}\sim  \frac{\psi_{r} (\sigma^{m-K}\omega)}{g_{r}
	(\omega)}\frac{\sum_{y\in S_{m}}G_{r} (1,x)^{2}}{\exp \{(m-K) 
	\text{Pressure} (2\varphi_{r}) \}} 
\end{equation}

Since $\Sigma^{*}\cup \Sigma $ is compact, any subsequence of
$\lambda^{*}_{r,m}$ must contain a subsequence $\lambda^{*}_{r,m_{n}}$
that converges weakly. The weak limit of any such subsequence must be
a probability measure with support contained in $\Sigma$, because the support of
$\lambda^{*}_{r,m}$ is $\Sigma^{m-K}$. We claim that there can be only one
possible weak limit $\lambda^{*}_{r}$, and that
\eqref{eq:sphereAsymptotics} must hold. To see this, suppose that
$\lambda^{*}_{r}$ is a weak limit of some subsequence
$\lambda^{*}_{r,m_{n}}$; then for any fixed cylinder set $\Sigma^{l}
(\omega^{(l)})$, \eqref{eq:measureRatio} implies that
\begin{equation}\label{eq:subsequential}
	\lambda^{*}_{r} (\Sigma^{l}(\omega^{(l)}))=\lim_{n \rightarrow
	\infty}
	 \frac{	g_{r} (\omega^{(l)}) \exp \{(m_{n}-K) 
	\text{Pressure} (2\varphi_{r})\}}{\sum_{y\in S_{m_{n-K}}}G_{r} (1,x)^{2}}
	\int_{\Sigma^{l} (\omega^{(l)})} \psi_{r}
	(\sigma^{m_{n}-K}\xi)\,d\mu_{r} (\xi).
\end{equation}
But because $\mu_{r}$ is mixing (cf. \cite{bowen}, Prop.~1.14),
\[
	\lim_{ m \rightarrow \infty}	\int_{\Sigma^{l}
	(\omega^{(l)})} \psi_{r} (\sigma^{m-K}\xi )^{-1} \,d\mu_{r} (\xi)
	=\mu_{r} (\Sigma^{l}(\omega^{(l)})) \int_{\Sigma} \psi_{r}
	(\xi)^{-1}\, d\mu_{r} (\xi).
\]
Therefore, the convergence in \eqref{eq:subsequential} holds along the
entire sequence of positive integers, not merely along the
subsequence. Hence, the measure $\lambda^{*}_{r}$ is uniquely
determined, and \eqref{eq:sphereAsymptotics} holds. Finally, the
convergence \eqref{eq:subsequential} and the definition of a Gibbs
state imply that $\lambda^{*}_{r}$ is absolutely continuous with
respect to $\mu_{r}$.
\end{proof}

\begin{note}\label{note:theta}
Virtually the same argument shows that for any $\theta \in \zz{R}$, as
$m \rightarrow \infty$,
\[
	\sum_{x\in S_{m}} G_{r} (1,x)^{\theta }
	\sim  C \exp \bigg\{m \text{\rm Pressure} (\theta  \varphi_{r}) \bigg\}.
\]
The result \eqref{eq:sphereAsymptotics} implies that $\text{Pressure}
(2\varphi_{r})\leq 0$, and Note~\ref{note:greenOnSphere} implies that
$\text{Pressure} (\varphi_{r})>0$ for all $r\in (1,R]$.  Since
$\text{Pressure} (\theta \varphi_{r})$ varies continuously with
$\theta$, it follows that for each $r\in (1,R]$ there exists $\theta
\in (1,2]$ such that $\text{Pressure} (\theta \varphi_{r})=0$.  It can
also be shown that the convergence of the sums is uniform in $r$ for
$r\in [1,R]$.
\end{note}

By Proposition~\ref{proposition:uniformErgodicity}, ergodic averages
of any H\"{o}lder continuous function will converge a.s. under any
$\mu_{r}$ to their $\mu_{r}-$integrals, and the probabilities of large
deviations will be small uniformly for $r\in [1,R]$. Since $\lambda^{*}_{r}\ll
\mu_{r}$, with Radon-Nikodym derivative varying continuously with $r$
in the H\"{o}lder norm, the uniformity transfers to the measures
$\lambda^{*}_{r}$.

\begin{corollary}\label{corollary:ergodic}
Let $g:\Sigma\cup \Sigma^{*} \rightarrow \zz{R}$ be any H\"{o}lder continuous
function. For each $x\in S_{m}$, let $\omega^{m-K}$ be the word of
length $m-K$ in $\Sigma^{*}$ that corresponds to the path
$\varrho^{x}$, and set $n=m-K$. Then
\begin{equation}\label{eq:ergodicTheorem}
		\lim_{m \rightarrow \infty}
	\lambda^{*}_{r,m}\left\{
	\omega \in \Sigma^{n}\,:\, \Bigl \lvert
	     n^{-1}\sum_{j=1}^{n} g\circ \sigma^{j} (\omega^{n})
	     -\int g\,d\mu_{r}\Bigr \rvert >\varepsilon 
	 \right\} =0.
\end{equation}
and for each $\varepsilon >0$ the convergence is uniform in $r\in [1,R]$.
\end{corollary}

\begin{proof}
For any infinite sequence $\omega \in \Sigma$ denote by
$\omega^{(n)}\in \Sigma^{n}$ the sequence of length $n$ consisting of
the first $n$ entries of $\omega$. Since $g$ is by hypothesis
H\"{o}lder continuous, there exist constants $C<\infty$ and
$0<\varrho<1$ such that for all $j\leq n$ and all $\omega \in \Sigma$,
\[
	|g (\sigma^{j}\omega) -g (\sigma^{j}\omega^{(n)})|\leq
	C\varrho^{n-j}. 
\]
Hence, with $C'=C/ (1-\varrho)$, for every $n\geq 1$,
\[
	\Bigl \lvert n^{-1}\sum_{j=1}^{n} g\circ \sigma^{j} (\omega)
	-	n^{-1}\sum_{j=1}^{n} g\circ \sigma^{j}
	(\omega^{(n)})\Bigr \rvert \leq \frac{C'}{n}.
\]

Now by equation \eqref{eq:cylinderProbs}, if $W$ is a $\Sigma -$valued
random variable with distribution $\lambda_{r}$ then for each $m>
K$ the truncation $W^{(n)}=W^{( m-K)}$ has distribution
$\lambda^{*}_{r,m}$. The result therefore follows from
Proposition~\ref{proposition:uniformErgodicity} and the preceding
paragraph.
\end{proof}

The ergodic average in \eqref{eq:ergodicTheorem} is expressed as an
average over the orbit of the path $\omega^{(n)}=\varrho^{x}$ in the
Cannon automaton, but it readily translates to an equivalent statement
for ergodic averages along the geodesic segment $L=L (1,x)$. Observe
that each vertex $y\in L$ disconnects $L$ into two geodesic segments
$L^{+}=L^{+}_{y}$ and $L^{-}=L^{-}_{y}$, where $L^{+}$ is the segment
of $L$ from $y$ to $x$, and $L^{-}$ is the segment of $L$ from $y$ to
$1$. These paths determine determine finite reduced words $e^{+}=e^{+}
(y)$ and $e^{-}=e^{-} (y)$ in the group generators $A$ (recall that
each oriented edge $(u,v)$ of the Cayley graph is labelled by the
generator $u^{-1}v$). The word $e^{+} (y)$ and the reversal of the
word $e^{-} (y)$ are both elements of $\Sigma^{*}$; thus, the concatenation
$e^{-}e^{+}$ can be viewed as an element of the compact metric
space $\tilde{\Sigma }^{*}\cup  \tilde{\Sigma}$, where
$\tilde{\Sigma}$ is the set of all bi-infinite paths 
and $\tilde{\Sigma}^{*}$  the set of finite or semi-infinite paths
in $\mathcal{R}$.

\begin{corollary}\label{corollary:ergodicCorollary}
Let $f:\tilde{\Sigma }^{*}\cup  \tilde{\Sigma} \rightarrow \zz{R}$ be
a H\"{o}lder continuous function, and let $g=f\circ \alpha^{-1}$ be
its pullback to the space of two-sided paths in the Cannon
automaton. Then for each $\varepsilon >0$,
\begin{equation}\label{eq:erg2}
			\lim_{m \rightarrow \infty}
	\lambda_{r,m}\left\{
	x\in S_{m}\,:\, \Bigr \lvert
	     m^{-1}\sum_{y\in  L (1,x)} f (e^{+} (y),e^{-} (y))
	     -\int g\,d\mu_{r}\Bigr \rvert >\varepsilon 
	 \right\} =0.
\end{equation}
and the convergence is uniform in $r\in [1,R]$.
\end{corollary}

\subsection{Proofs of Lemma~\ref{lemma:gibbsCharacterization} and
Proposition~\ref{proposition:uniformErgodicity}}\label{ssec:proofs}

 \begin{proof}[Proof of Lemma~\ref{lemma:gibbsCharacterization}] In
view of the representation \eqref{eq:gibbsStateDecomp} of the Gibbs
state $\mu =\mu_{\varphi }$, it suffices to prove the assertions of
the lemma with $\mu $ replaced by $\nu =\nu_{\varphi}$ and
$h_{\varphi}$ by the constant function $1$, where $\nu$ is
the Perron-Frobenius eigenmeasure of the adjoint Ruelle operator
$\mathcal{L}^{*}$. The Ruelle operator
$\mathcal{L}=\mathcal{L}_{\varphi }$ is defined as follows for
continuous functions $g:\Sigma^{*}\cup \Sigma \rightarrow \zz{R}$:
\[
	\mathcal{L}g (\omega) = \sum_{\xi : \sigma (\xi )=\omega}
		     \exp \{\varphi (\xi) \}g (\xi)
	\quad \Longrightarrow \quad 
	\mathcal{L}^{n}g (\omega) = \sum_{\xi : \sigma^{n} (\xi )=\omega}
		     \exp \{S_{n}\varphi (\xi) \}g (\xi).
\] 
Denote by $\lambda =\exp \{\text{Pressure} (\varphi) \}$  the
Perron-Frobenius eigenvalue of $\mathcal{L}$. 
Using the fact that $\nu =\lambda^{-1}\mathcal{L}^{*}\nu =\lambda^{-n}
(\mathcal{L}^{*})^{n}\nu$, we have (with $\mathbf{1}_A$ denoting the
indicator function of $A$, and using inner product notation
$\xclass{g,\nu}=\int g\, d\nu$) 
\begin{align*}
	\nu (\Sigma^{n} (\omega))&=\xclass{\mathbf{1}_{\Sigma^{n} (\omega)},\nu}\\
    &=\lambda^{-n}\xclass{\mathbf{1}_{\Sigma^{n}
    (\omega)},(\mathcal{L}^{*})^{n}\nu}\\
		    &= \lambda^{-n}\xclass{\mathcal{L}^{n}\mathbf{1}_{\Sigma^{n}
    (\omega)},\nu} .
\end{align*}
Using the definition of $\mathcal{L}$, we obtain
\begin{align*}
	\nu (\Sigma^{n}
	(\omega))&=\lambda^{-n}\int_{\Sigma}\left(\sum_{\zeta \,:\,
	\sigma^{n}  (\zeta)=\xi} \exp \{S_{n}\varphi (\zeta) \}\mathbf{1}_{\Sigma^{n}
    (\omega)} (\zeta) \right)  \,d\nu (\xi )\\
    &= \lambda^{-n} \exp\{S_n \varphi (\omega)\} \int_{\Sigma}\left(\sum_{\zeta \,:\,
	\sigma^{n}  (\zeta)=\xi} \exp \{S_{n}\varphi
	(\zeta)-S_{n}\varphi (\omega) \}\mathbf{1}_{\Sigma^{n} 
    (\omega)} (\zeta) \right)  \,d\nu (\xi ).
\end{align*}
The integrand is nonzero only for those sequences $\zeta\in \Sigma $
such that $\zeta_{i}=\omega_{i}$ for $1\leq i\leq n$, and consequently
only for those $\xi \in \Sigma$ such that the concatenation
$\omega^{(n)}\xi$ is an element of $\Sigma$. (Here $\omega (n)$ is the
string of length $n$ consisting of the first $n$ entries of
$\omega$. The concatenation $\omega^{(n)}\xi\in \Sigma$ if and only if
$\omega_{n} \mapsto \xi_{1}$ is an allowable transition in the
subshift $(\Sigma ,\sigma)$).
Because $\varphi:\Sigma \rightarrow
\zz{R}$ is H\"{o}lder continuous and depends only on the forward
entries of its entries, there is a H\"{o}lder continuous function
$\psi  :\tilde{\Sigma} \rightarrow \zz{R}$ such that 
\[
	\lim_{n \rightarrow \infty}\int_{\Sigma}\left(\sum_{\zeta \,:\,
	\sigma^{n}  (\zeta)=\xi} \exp \{S_{n}\varphi
	(\zeta)-S_{n}\varphi (\omega) \}\mathbf{1}_{\Sigma^{n} 
    (\omega)} (\zeta) \right)  \,d\nu (\xi ) -\psi  (\sigma^{n}\omega)=0
\]
and such that the error is bounded by $Cr^{n}$ for some $C>0$ and
$0<r<1$. This implies
\eqref{eq:gibbsCharacterization}--\eqref{eq:gibbsCharacterizationRate}.
\end{proof}

\begin{proof}
[Proof of Proposition~\ref{proposition:uniformErgodicity}] It suffices
to consider the special case of functions $\varphi_{r}$ and $g_{r}$
that depend only on $\omega \in \tilde{\Sigma}$ through the
coordinates $\omega_{1}\omega_{2}\dotsb$.  (This follows from the fact
that any H\"{o}lder continuous function $f$ on $\tilde{\Sigma}$ is
cohomologous to a H\"{o}lder continuous function on $\Sigma$, and that
the implied coboundary can be chosen so as to vary continuously with
$f$. See \cite{bowen}, Lemma~1.6 and its proof.)

Ruelle's Perron-Frobenius theorem asserts that as an operator on 
the space of H\"{o}lder continuous functions (relative to a given
exponent) the operator $\mathcal{L}_{\varphi}$ has a simple, positive
eigenvalue $\lambda_{\varphi}$, that the corresponding eigenfunction
$h_{\varphi}$ is strictly positive, and that the remainder of the
spectrum is contained in a disc of radius $<\lambda_{\varphi}$. (Bowen
\cite{bowen} does not state the result containing the residual
spectrum, although his argument essentially proves it. See
\cite{parry-pollicott} for the stronger version.) It then follows by
standard  arguments in regular perturbation theory (cf. \cite{kato},
especially ch.~6) that $\lambda_{\varphi}, h_{\varphi},$ and
$\nu_{\varphi}$ (the eigenmeasure for the adjoint operator
$\mathcal{L}_{\varphi}^{*}$) vary smoothly with $\varphi$,
assuming that $h_{r}$ and $\nu_{r}$ are normalized so that 
$\nu_{r}$ and $h_{r}\nu_{r}$ both have total mass $1$.Furthermore,
the residual spectral radius remains uniformly bounded away from
$\lambda_{\varphi }$ for $\varphi$ in any compact set. It follows that if
$r\mapsto g_{r}$ and $r\mapsto \varphi_{r}$ are continuous relative to
the H\"{o}lder norm then $\xclass{g_{r},\mu_{r}}$ varies continuously
with $r$. Therefore, there is no loss of generality in assuming that
$\int g_{r}\,d\mu_{r}=0$  for all $r\in [1,R]$. Moreover, because
$\varphi_{r}$ can be replaced by $\varphi_{r}-\log h_{r}\circ \sigma
+\log h_{r}-\log \lambda_{r}$, there  is no loss of generality in
assuming that $\mathcal{L}_{\varphi_{r}}1=1$ for all $r$. In
particular, $\lambda_{r}=1$, $h_{r}\equiv 1$, and
$\mu_{r}=\nu_{r}$.

Next, we express the moment generating function of $S_{n}g_{r}$ in
terms of the Ruelle operator $\mathcal{L}_{\varphi_{r}+\theta
g_{r}}$. Since $\nu_{r}=\mu_{r}$, and since
$\mathcal{L}^{*}_{\varphi_{r}}\nu_{r}=\nu_{r}$, it follows that for
each $\theta \in \zz{R}$,
\begin{align*}
	\xclass{e^{\theta S_{n}g_{r}},\mu_{r}}
	 &=\xclass{e^{\theta S_{n}g_{r}},\nu_{r}}\\
	&=\xclass{e^{\theta S_{n}g_{r}},(\mathcal{L}^{*}_{\varphi_{r}})^{n}nu_{r}}\\
	&=\xclass{\mathcal{L}_{\varphi_{r}}^{n}e^{\theta S_{n}g_{r}},\nu_{r}}\\
	&=\xclass{\mathcal{L}_{\varphi_{r}+\theta g_{r}}^{n}1,\nu_{r}}.
\end{align*}
 The operators $\mathcal{L}_{\varphi_{r}+\theta g_{r}}$ vary
analytically with $\theta$, as do their lead eigenvalue,
eigenfunction, and eigenmeasure, and so by Ruelle's Perron-Frobenius
theorem and regular perturbation theory,
\[
	\mathcal{L}_{\varphi_{r}+\theta g_{r}}^{n}1
	\sim \lambda_{\varphi_{r}+\theta g_{r}}^{n}h_{\varphi_{r}+\theta g_{r}}
\]
uniformly for $r\in [1,R]$ and for $\theta$ in any compact
neighborhood of $0$. Since $h_{\varphi_{r}+\theta g_{r}}$ varies
continuously with $r$ and $\theta$ (relative to the H\"{o}lder norm,
and hence also the sup norm), it follows that for any $\delta >0$
there exists $C=C_{\delta}<\infty$  such that 
\[
	\xclass{e^{\theta S_{n}g_{r}},\mu_{r}}\leq
	C\lambda_{\varphi_{r}+\theta g_{r}}^{n} 
	=C\exp \{n\text{Pressure} (\varphi_{r}+\theta g_{r})\}
\]
for all $n\geq 1$, all $r\in [1,R]$, and all $\theta \in [-\delta
,\delta]$. 

The proof of the proposition is now completed by applying the
Markov-Chebyshev inequality, which implies that for any
$\varepsilon>0$ and $\theta >0$,
\[
	\mu_{r}\{S_{n}g_{r}\geq n\varepsilon \}\leq 
	e^{-n \theta \varepsilon} \int e^{-\theta S_{n}g_{r}}
	\,d\mu_{r}
	\leq C \exp \{-n\theta  \varepsilon +n\text{Pressure}
	(\varphi_{r}+\theta g_{r})\} .
\]
By elementary calculations (see, e.g.,
\cite{parry-pollicott}, Propositions 4.10-4.11),
\[
	\frac{d \text{Pressure} (\varphi_{r}+\theta g_{r})}{d\theta } 
	= \int g_{r} \,d\mu_{r} =0
	\quad \text{and} \quad 
	\frac{d^{2}\text{Pressure} (\varphi_{r}+\theta g_{r})}{d\theta
	^{2}} \geq 0.
\]
Consequently, for any $\varepsilon >0$, if $\theta >0$ is sufficiently
small then the Markov-Chebyshev bound is exponentially decaying in
$n$, uniformly for $r\in [1,R]$. This proves half of
\eqref{eq:uniformLD}; the other half is obtained by replacing
$g_{r}$ by $-g_{r}$.

\end{proof}

\section{Evaluation of the pressure at
$r=R$}\label{ssec:pressureEvaluation}

Proposition~\ref{proposition:absolutelyCont} implies that the sums
$\sum_{y\in S_{m}}G_{R} (1,y)^{2}$ grow or decay sharply exponentially
at exponential rate $\text{\rm Pressure}(2\varphi_{R})$. Consequently, to prove
the relation \eqref{eq:backscatterA} of Theorem~\ref{theorem:2} it
suffices to prove that this rate is $0$.

\begin{proposition}\label{proposition:pressureEqualsZero}
$\text{\rm Pressure}(2\varphi_{R})=0$.
\end{proposition}

The second assertion \eqref{eq:lp} of
Theorem~\ref{theorem:2} also follows from
Proposition~\ref{proposition:pressureEqualsZero},  by the main result
of \cite{lalley:renewal}. (If it could be shown that the cocycle
$\varphi_{R}$ defined by \eqref{eq:defCocycle} above is
\emph{nonlattice} in the sense of \cite{lalley:renewal}, then the
result\eqref{eq:lp} could be strengthened from $\asymp$ to $\sim$.)

The remainder of this section is devoted to the proof of
Proposition~\ref{proposition:pressureEqualsZero}.  The first step,
that $\text{\rm Pressure}(2\varphi_{R})\leq 0$, is a consequence of
the differential equations \eqref{eq:GPrime}.  These imply the
following.

\begin{lemma}\label{corollary:belowR}
For every $r<R$,
\begin{align}\label{belowR}
	\text{\rm Pressure} (2\varphi_{r})&<0, \quad \text{and so}\\
\label{eq:atR}
	\text{\rm Pressure} (2\varphi_{R})&\leq 0.
\end{align}
\end{lemma}

\begin{proof}
For $r<R$ the Green's function $G_{r} (1,1)$ is analytic in $r$, so
its derivative must be finite. Thus, by
Proposition~\ref{proposition:GPrime}, the sum $\sum_{x\in
\Gamma}G_{r}(1,x)^{2}$ is finite. (The last term $r^{-1}G_{r (1,1)}$
in equation \eqref{eq:GPrime} remains bounded as $r \rightarrow R-$
because $G_{R} (1,1,)<\infty$.)
Proposition~\ref{proposition:absolutelyCont}  therefore implies that
$\text{\rm Pressure} (2\varphi_{r})$ must be negative.  Since
$\text{Pressure} (\varphi)$ varies continuously in $\varphi$, relative
to the sup norm, \eqref{eq:atR} follows.
\end{proof}

\begin{proof}
[Proof of Proposition~\ref{proposition:pressureEqualsZero}.]  To
complete the proof it  suffices, by the preceding lemma, to
show that $\text{Pressure} (2\varphi_{R})$ cannot be negative. In view
of Proposition~\ref{proposition:absolutelyCont}, this is equivalent to
showing that $\sum_{x\in S_{m}}G_{R} (1,x)^{2}$ cannot decay
exponentially in $m$.  This will be accomplished by proving that
exponential decay of $\sum_{x\in S_{m}}G_{R} (1,x)^{2}$ would force
\begin{equation}\label{eq:subcriticality}
	G_{r} (1,1)<\infty  \quad \text{for some} \; r>R,
\end{equation}
which is impossible since $R$ is the radius of convergence of the
Green's function.

To prove \eqref{eq:subcriticality}, we will use the branching random
walk interpretation of the Green's function discussed in
sec.~\ref{ssec:brw}.\footnote{Logically this is unnecessary --- the
argument has an equivalent formulation in terms of weighted paths,
using \eqref{eq:greenByPath} --- but the branching random walk
interpretation seems more natural.}  Recall that a branching random
walk on the Cayley graph $G^{\Gamma}$ is specified by an offspring
distribution $\mathcal{Q}$; assume for definiteness that this is the
Poisson distribution with mean $r>0$. At each step, particles first
produce offspring particles according to this distribution,
independently, and then each of these particles jumps to a randomly
chosen neighboring vertex. If the mean of the offspring distribution
is $r>0$, and if the branching random walk is initiated by a single
particle at the root $1$, then the mean number of particles located at
vertex $x$ at time $n\geq 1$ is $r^{n}P^{1}\{X_{n}=x \}$. Thus, in
particular, $G_{r} (1,1)$ equals the expected total  number of particle
visits to the root vertex $1$. The strategy is to show that if 
 $\sum_{x\in S_{m}}G_{R} (1,x)^{2}$ decays exponentially in $m$,
then for some $r>R$ the branching random walk remains
\emph{subcritical}, that is, the expected total number of particle
visits to $1$ is finite. 

Recall that the Poisson distribution with mean $r>R$ is the
convolution of  Poisson distributions with means $R$ and $\varepsilon
:=r-R$, that is, the result of adding independent random variables
$U,V$ with distributions Poisson-$R$ and Poisson-$\varepsilon$ is a
random variable $U+V$ with distribution Poisson-$r$. Thus, each 
reproduction step in the branching random walk can be done by making
independent draws $U,V$ from the Poisson-$R$ and Poisson-$\varepsilon$
distributions. Use these independent draws to assign \emph{colors}
$k=0,1,2,\dotsc$ to the particles according to the following rules:

\begin{enumerate}
\item [(a)] The ancestral particle at vertex $1$ has color $k=0$.
\item [(b)]  Any offspring resulting from a $U-$draw has the
same color as its parent.
\item [(c)] Any offspring resulting from a $V-$draw has color equal to
$1+$the color of its parent.
\end{enumerate}

\begin{lemma}\label{lemma:colors}
For each $k=0,1,2,\dotsc$, the expected number of visits to the vertex
$y$ by particles of color $k$ is 
\begin{equation}\label{eq:colors}
	v_{k} (y) =\varepsilon^{k} \sum_{x_{1},x_{2},\dotsc x_{k}\in
	\Gamma} G_{R} (1,x_{1})\left( \prod_{i=1}^{k-1} G_{R}
	(x_{i},x_{i+1})\right) G_{R} (x_{k},y)	.
\end{equation}
\end{lemma}

\begin{proof}
By induction on $k$. First, particles of color $k=0$ reproduce and
move according to the rules of a branching random walk with offspring
distribution Poisson-$R$, so the expected number of visits to vertex
$y$ by particles of color $k=0$ is $G_{R} (1,y)$, by
Proposition~\ref{proposition:brw}. This proves \eqref{eq:colors} in
the case $k=0$. Second, assume that the assertion is true for color
$k\geq 0$, and consider the production of particles of color
$k+1$. Such particles are produced only by particles of color $k$ or
color $k+1$. Call a particle a \emph{pioneer} if its color is
different from that of its parent, that is, if it results from a
$V-$draw. Each pioneer of color $k+1$ engenders its own branching
random walk of  descendants with color $k+1$;  the offspring distribution for
this branching random walk is the Poisson-$R$ distribution. Thus, for
a pioneer born at site $z\in \Gamma$, the expected number of visits to
$y$ by its color--$( k+1)$ descendants is $G_{R} (z,y)$. Every particle of color
$k+1$ belongs to the progeny of one and only one pioneer;
consequently, the expected number of visits to $y$ by particles of
color $k+1$ is
\[
	\sum_{z\in \Gamma }u_{k+1} (z) G_{R} (z,y),
\]
where $u_{k+1} (z)$ is the expected number of pioneers of color $k+1$
born at site $z$ during the evolution of the  branching process. But
since  pioneers of color $k+1$ must be children of parents of color
$k$, and since for any particle  the expected number of children of
different color is $\varepsilon$, it follows that 
\[
	u_{k+1} (z) =\varepsilon v_{k} (z).
\]
Hence, formula \eqref{eq:colors} for $k+1$ follows by the induction hypothesis.
\end{proof}

Recall that our objective is to show that if $\sum_{x\in S_{m}}G_{R}
(1,x)^{2}$ decays exponentially in $m$ then 
$G_{r} (1,1)<\infty$ for some $r=R+\varepsilon >R$. The branching
random walk construction exhibits $G_{r} (1,1)$ as the expected total
number of particle visits to the root vertex $1$, and this is the sum
over $k\geq 0$ of the expected number $v_{k} (1)$ of visits by
particles of color $k$. Thus, to complete the proof of
Proposition~\ref{proposition:pressureEqualsZero} it suffices, by
Lemma~\ref{lemma:colors}, to show that for some $\varepsilon >0$,
\begin{equation}\label{eq:epsilon}
	\sum_{k=0}^{\infty} \varepsilon^{k}\sum_{x_{1},x_{2},\dotsc x_{k}\in
	\Gamma}G_{R} (1,x_{1})\left( \prod_{i=1}^{k-1} G_{R} (x_{i},x_{i+1}) \right)
	G_{R} (x_{k},1)<\infty .
\end{equation}
This follows directly from the next lemma.
\end{proof}

\begin{lemma}\label{lemma:snapback}
Assume that Ancona's inequalities \eqref{eq:ancona} hold at the
spectral radius $R$ with a constant $C_{R}<\infty$. If the sum  $\sum_{x\in
S_{m}}G_{R} (1,x)^{2}$ decays exponentially in $m$, then there exist
constants $\delta >0$ and $C,\varrho  <\infty$ such that for
every $k\geq 1$,
\begin{equation}\label{eq:snapback}
	\sum_{x_{1},x_{2},\dotsc x_{k}\in
	\Gamma} G_{R} (1,x_{1})\left( \prod_{i=1}^{k-1} G_{R} (x_{i},x_{i+1}) \right)
	(1+\delta)^{|x_{k}|}G_{R} (x_{k},1)\leq C\varrho ^{k}.
\end{equation}
Here $|y|=d (1,y)$ denotes the distance of $y$ from the root $1$ in
the word metric.
\end{lemma}

\begin{proof}
Denote by $H_{k} (\delta)$ the left side of \eqref{eq:snapback}; the
strategy will be to prove by induction on $k$ that for sufficiently
small $\delta >0$ the ratios $H_{k+1} (\delta)/H_{k} (\delta)$ remain
bounded as $k \rightarrow \infty$.  Consider first the sum $H_{1}
(\delta)$: by the hypothesis that $\sum_{x\in S_{m}}G_{R} (1,x)^{2}$
decays exponentially in $m$ and the symmetry $G_{r} (x,y)=G_{r} (y,x)$
of the Green's function, for all sufficiently small $\delta>0$
\begin{equation}\label{eq:H1}
	H_{1} (\delta ):=\sum_{x\in \Gamma} G_{R} (1,x)^{2} (1+\delta)^{|x|}<\infty .
\end{equation}

Now consider the ratio $H_{k+1} (\delta)/H_{k} (\delta)$ . Fix
vertices $x_{1},x_{2},\dotsc ,x_{k}$, and for an arbitrary vertex
$y=x_{k+1}\in \Gamma $, consider its position \emph{vis a vis} the
geodesic segment $L=L (1,x_{k})$ from the root vertex $1$ to the
vertex $x_{k}$. Let $z\in L$ be
the vertex on $L$ nearest $y$ (if there is more than one, choose
arbitrarily).  By the triangle inequality,
\[
	|y|\leq |z|+d (z,y) .
\]
Because
the  group $\Gamma$ is word-hyperbolic, all geodesic triangles --- in
particular, any  triangle whose sides consist of geodesic segments 
from $y$ to $z$, from $z$ to $x_{k}$, and from $x_{k}$ to $y$, or any
triangle  whose sides consist of geodesic segments 
from $y$ to $z$, from $z$ to $1$, and from $1$ to $y$--- are
$\Delta -$thin, for some $\Delta<\infty$ (cf. \cite{gromov} or
\cite{benakli-kapovich}). 
Hence, any geodesic segment from
$x_{k}$ to $y$ must pass within distance $8\Delta$ of the vertex
$z$. Therefore, by the Harnack and Ancona inequalities
\eqref{eq:harnack} and \eqref{eq:ancona}, for some constant
$C_{*}=C_{R}C_{\text{Harnack}}^{32\Delta}<\infty$ independent of $y,x_{k}$,
\begin{align*}
	G_{R} (y,1)&\leq C_{*}G_{R} (y,z)G_{R} (z,1) \quad \text{and}\\
	G_{R} (y,x_{k})&\leq C_{*}G_{R} (y,z)G_{R} (z,x_{k}).
\end{align*}
On the other hand, by the log-subadditivity of the Green's function,
\[
	G_{R} (1,z)G_{R} (z,x_{k})\leq G_{R} (x_{k},1).
\]
It now follows that
\begin{align*}
	(1+\delta)^{|y|} G_{R} (x_{k},y)G_{R} (y,1)
	&\leq C_{*}^{2}(1+\delta)^{|z|+d (z,y)} G_{R} (z,x_{k}) G_{R} (z,y) G_{R} (y,z)G_{R} (z,1)\\
	&\leq C_{*}^{2}(1+\delta)^{|z|+d (z,y)} G_{R} (x_{k},1)G_{R}(z,y)^{2}.
\end{align*}
Denote by $\Gamma (z)$ the set of all vertices $y\in \Gamma$ such that
$z$ is a closest vertex to $y$ in the geodesic segment $L$. Then for
each $z\in L$,
\[
	\sum_{y\in \Gamma (z)} (1+\delta)^{d (z,y)}G_{R} (z,y)^{2}\leq 
	\sum_{y\in \Gamma } (1+\delta)^{|y|} G_{R} (1,y)^{2}=H_{1} (\delta).
\]
Finally, because $L$ is a geodesic segment from $1$ to $x_{k}$ there
is precisely one vertex $z\in L$ at distance $n$ from $x_{k}$ for
every integer $0\leq n\leq |x_{k} |$, so $\sum_{z\in L} (1+\delta)^{|z|} \leq C_{\delta}
(1+\delta)^{|x_{k}|}$ where $C_{\delta}= (1+\delta)/ (2+\delta)$. Therefore, 
\[
	H_{k+1} (\delta) \leq C_{*}^{2}C_{\delta }H_{1} (\delta)H_{k} (\delta).
\]
\end{proof}

\section{Critical Exponent of the Green's function at the Spectral Radius}\label{sec:criticalExp} 

\subsection{Reduction to Ergodic
Theory}\label{ssec:strategy} For ease of exposition I will consider
only the case $x=y=1$ of Theorem~\ref{theorem:criticalExponent}; the
general case can be done in the same manner.  The system of
differential equations \eqref{eq:GPrime} implies that the growth of the
derivative $dG_{r} (1,1)/dr$ as $r \rightarrow R-$ is controlled by
the growth of the quadratic sums $\sum_{x\in \Gamma}G_{r}
(1,x)^{2}$. To show that the Green's function has a square root
singularity at $r=R$, as asserted in \eqref{eq:criticalExponent}, it
will suffice to show that the (approximate) derivative behaves as
follows as $r \rightarrow R-$:

\begin{proposition}\label{proposition:eta}
For some $0<C<\infty$,
\begin{equation}\label{eq:eta}
	\eta (r):= \sum_{x\in \Gamma} G_{r} (1,x)^{2}\sim C/\sqrt{R-r} 
	\quad \text{as} \;r \rightarrow R-.
\end{equation}
\end{proposition}

This will follow from Corollary~\ref{corollary:etaDE} below.  The key
to the argument is that the
growth of $\eta (r)$ as $r \rightarrow R-$ is related by
Proposition~\ref{proposition:absolutelyCont} to that of
Pressure$(2\varphi_{r})$: in particular, 
Proposition~\ref{proposition:pressureEqualsZero} implies that $\eta
(r) \rightarrow \infty$ as $r \rightarrow R-$, so the dominant
contribution to the sum \eqref{eq:eta} comes from vertices $x$ at
large distances from the root vertex $1$. Consequently, by equation
\eqref{eq:sphereAsymptotics}, 
\begin{equation}\label{eq:eta-Pressure}
	\eta (r)\sim C (R,2) / (1-\exp \{\text{\rm Pressure}
	(2\varphi_{r}) \})
	\quad \text{as} \;\; r \rightarrow R-.
\end{equation}

To analyze the behavior of $\eta (r)$ (or equivalently that of
Pressure$(2\varphi_{r})$) as $r \rightarrow R-$, we
use the differential equations \eqref{eq:GPrime} to express the
derivative of $\eta (r)$ as
\begin{equation}\label{eq:derivEta}
	\frac{d\eta}{dr}=\sum_{x\in \Gamma} \left\{ \sum_{y\in \Gamma}
				    2r^{-1} G_{r} (1,x) G_{r}
				    (1,y)G_{r} (y,x)\right\}
				    -2r^{-1}G_{r} (1,x)^{2}.
\end{equation}
(Note: The implicit interchange of $d/dr$ with an infinite sum is
justified here because the Green's functions $G_{r} (u,v)$ are defined
by power series with nonnegative coefficients.) For $r \approx R$, the
sum $\sum_{x\in \Gamma }$ is once again dominated by those vertices
$x$ at large distances from the root $1$. Because the second term
$2r^{-1}G_{r} (1,x)^{2}$ in \eqref{eq:derivEta} remains bounded as $r
\rightarrow R-$, it is asymptotically negligible compared to the first
term $\sum_{x}\sum_{y}$ and so we can ignore it in proving \eqref{eq:eta}.

The strategy for dealing with the inner sum $\sum_{y\in \Gamma}$ in
\eqref{eq:derivEta} will be similar to that used in the proof of
Lemma~\ref{lemma:snapback} above. For each $x$, let $L=L (1,x)$ be the unique
geodesic segment from the root to $x$ that corresponds to a path in
the Cannon automaton, and partition the sum $\sum_{y\in \Gamma }$
according to the nearest vertex $z\in L$:
\begin{equation}\label{eq:sumPartition}
	\sum_{y\in \Gamma}=\sum_{z\in L}\sum_{y\in \Gamma (z)}
\end{equation}
where $\Gamma (z)$ is the set of all vertices $y\in \Gamma$ such that
$z$ is a closest vertex to $y$ in the geodesic segment $L$. (If for
some $y$ there are several vertices $z_{1},z_{2},\dotsc$ on $L$ all
closest to $y$, put $y\in \Gamma (z_{i})$ only for the vertex $z_{i}$
nearest to the root $1$.) By the log-subadditivity of the Green's
function and Theorem~\ref{theorem:1} (the Ancona
inequalities) there exists a constant $C<\infty$ independent of $1\leq
r\leq R$ such that for all choices of $x\in \Gamma$, $z\in L (1,x)$,
and $y\in \Gamma (z)$,
\begin{align*}
	G_{r} (1,x) G_{r}(1,y)G_{r} (y,x)& \leq CG_{r} (1,z)^{2}G_{r}
	(z,x)^{2} G_{r} (z,y)^{2} \\
\notag 	&\leq CG_{r} (1,x)^{2}G_{r} (z,y)^{2};
\end{align*}
consequently, for each $x\in \Gamma$,
\begin{align}\label{eq:secondIneq}
		\sum_{y \in \Gamma }G_{r} (1,x) G_{r}(1,y)G_{r} (y,x)
 		&\leq \sum_{z\in L (1,x)}\sum_{y\in \Gamma (z)}
			     CG_{r} (1,x)^{2}G_{r} (z,y)^{2}\\
\notag 		&\leq \sum_{z\in L (1,x)}\sum_{y\in \Gamma}
			     CG_{r} (1,x)^{2}G_{r} (z,y)^{2}\\
\notag 		&=  CG_{r} (1,x)^{2} (|x|+1)\eta (r).
\end{align}
Proposition~\ref{proposition:ergodic2} below asserts that for large
$m$ and $r\approx R$ this inequality is in fact an approximate equality for
``most'' $x\in S_{m}$.  This  implies that for large $m$ the
contribution to the double sum in \eqref{eq:derivEta} with $|x|=m$ is
dominated by those $x$ that are ``generic'' for the probability
measure $\lambda_{r,m}$ on $S_{m}$ with density proportional to $G_{r}
(1,x)^{2}$ (cf. sec.~\ref{ssec:greenSpheres}).

\begin{proposition}\label{proposition:ergodic2}
For each $r\leq R$ and each $m=1,2,\dotsc$ let $\lambda_{r,m}$ be the
probability measure on the sphere $S_{m}$ with density proportional to
$G_{r} (1,x)^{2}$. There is a continuous, positive function $\xi (r)$
of $r\in [1,R]$ such that  each $\varepsilon >0$, and uniformly for $1\leq r\leq R$,
\begin{equation}\label{eq:ergodic}
	\lim_{m \rightarrow \infty}
	\lambda_{r,m}\left\{
	x\in S_{m}\,:\, \Bigr \lvert
	     \frac{1}{m}\sum_{y\in \Gamma} G_{r} (1,y)G_{r} (y,x) /G_{r} (1,x)
	     -\xi (r)\eta (r)\Bigr \rvert >\varepsilon 
	 \right\} =0.
\end{equation}
\end{proposition}

This will be deduced from Corollaries
\ref{corollary:ergodic}--\ref{corollary:ergodicCorollary} ---
see section~\ref{ssec:ergodic} below.  Given
Proposition~\ref{proposition:ergodic2},
Proposition~\ref{proposition:eta} and
Theorem~\ref{theorem:criticalExponent} follow easily, as we now show.

 \begin{corollary}\label{corollary:etaDE}
There exists a positive, finite constant $C$ such that as $r
\rightarrow R-$,
\begin{equation}\label{eq:etaDE}
	\frac{d\eta}{dr}\sim C \eta (r)^{3}
	\quad \text{as} \;r \rightarrow R-.
\end{equation}
Consequently,
\begin{equation}\label{eq:etaAsymptotics}
	\eta (r)^{-2}\sim C (R-r)/2.
\end{equation}
\end{corollary}

\begin{proof}
We have already observed that as $r$ near $R$, the dominant
contribution to the sum \eqref{eq:derivEta} comes from vertices $x$
far from the root.  Proposition~\ref{proposition:ergodic2} and the
uniform upper bound \eqref{eq:secondIneq} on ergodic averages imply
that as $r \rightarrow R-$,
\begin{align*}
	\frac{d\eta}{dr}\sim \sum_{x\in \Gamma} \sum_{y\in \Gamma}
				    2r^{-1} G_{r} (1,x) G_{r}
				    (1,y)G_{r} (y,x)
	&\sim 2R^{-1}\xi (R)\eta (r)\sum_{m=1}^{\infty}m\sum_{x\in
	S_{m}} G_{r} (1,x)^{2}\\
	&\sim C' \eta (r) / (1- \exp \{\text{\rm Pressure} (2\varphi_{r}) \})^{2}\\
	& \sim C \eta (r)^{3}
\end{align*}
for suitable positive constants $C,C'$. This
proves \eqref{eq:etaDE}. The relation \eqref{eq:etaAsymptotics}
follows directly from \eqref{eq:etaDE}.
\end{proof}

 \subsection{Proof of
Proposition~\ref{proposition:ergodic2}}\label{ssec:ergodic} This will
be accomplished by showing that the average in \eqref{eq:ergodic} can
be expressed approximately as an ergodic average of the form
\eqref{eq:erg2}, to which the result of
Corollary~\ref{corollary:ergodicCorollary} applies.  The
starting point is the decomposition \eqref{eq:sumPartition}. The inner
sum in \eqref{eq:sumPartition} is over the set $\Gamma (z)$ of
vertices $y$ for which $z$ is the nearest point on the geodesic
segment $L$. The following geometrical lemma implies that the set of
relative positions $z^{-1}y$, where $y\in \Gamma (z)$, depends only on
configuration of the geodesic segment $L$ in a bounded neighborhood of
$z$.

\begin{lemma}\label{lemma:geometry} 
If $L$ and $L'$ are geodesic segments both passing through the vertex
$z$, then denote by $\Gamma (z)$ and $\Gamma ' (z)$, respectively, the
sets of vertices $y$ such that $z$ is the nearest\footnote{Asssume
that the two geodesic segments $L,L'$ have the same orientation
relative to their common segment through $z$, so that in cases of
multiplicity ties are resolved the same way.}  vertex on $L$
(respectively, on $L'$) to $y$.  There exists $K<\infty$, independent
of $z$, $L$, and $L'$, so that if $L$ and $L'$ coincide in the ball of radius $K$
centered at $z$, then
\begin{equation}\label{eq:geometry}
	\Gamma (z)=\Gamma' (z').
\end{equation} 
\end{lemma}

\begin{proof}
This is a routine consequence of the thin triangle property.
\end{proof}

The next issue is the approximation in \eqref{eq:secondIneq}. For
this, the key is Corollary~\ref{corollary:bi} --- in particular,
inequality \eqref{eq:bi-ineq} --- which will ultimately justify
replacing $G_{r} (1,y)$ and $G_{r} (y,x)$ by the products
$G_{r}(1,z)G_{r} (z,y)$ and $G_{r} (y,z)G_{r} (z,x)$, respectively,
times suitable functions of $z$. The thin triangle property is
essential here, as it implies that, for $y\in \Gamma (z)$, any
geodesic segments from $y$ to $x$ or from $y$ to $1$ must pass within
distance $32\Delta$ of the point $z$ (see the proof of
Lemma~\ref{lemma:snapback}). Thus, if $z^{+} (y)$ and $z^{-} (y)$ are
the nearest vertices to $z$ on geodesic segments from $y$ to $x$ and
$y$ to $1$, respectively (with ties resolved by chronological
ordering), then both $z^{+} (y)$ and $z^{-} (y)$ are among the
vertices in the ball of radius $32 \Delta$ centered at $z$.  The
following lemma shows that the assignments $y\mapsto z^{\pm} (y)$ can
be made so as to depend only on the relative position of $y$ to $z$
and the configuration $(e^{+} (z),e^{-} (z))$ of the geodesic $L
(1,x)$ in a bounded neighborhood of $z$. (Recall
[Corollary~\ref{corollary:ergodic}] that $e^{+} (z)$ and $e^{-} (z)$
are the sequences of group generators corresponding to the steps of $L
(1,x)$ from $z$ forward to $x$ and from $z$ back to $1$,
respectively.)

\begin{lemma}\label{lemma:configuration}
Assume that the Cayley graph $G^{\Gamma}$ is \emph{planar}.
There exists $K<\infty$ such that the following is true. The
assignments $y\mapsto z^{\pm} (y)$ for $y\in \Gamma (z)$ on any
geodesic segment $L (1,x)$ can be made in such a way that the relative
positions 
\[
	z^{-1}z^{+} (y) \quad \text{and} \quad z^{-1}z^{-} (y)
\] 
depend only on the relative position $z^{-1}y$ of $y$ in the sector
$\Gamma (z)$ and the configuration $(e^{+} (z),e^{-} (z))$ of the
geodesic $L (1,x)$ restricted to the ball of radius $K$ centered at $z$.
\end{lemma}

\begin{proof}
The assertion is equivalent to this: For any geodesic segment $\gamma$
through $z$ and any vertex $y\in \Gamma (z)$, the nearest point $z^{+}_{n} (y)$
to $z$ on the geodesic segment from $y$ to a point $x_{n}$ on $\gamma$
outside the ball of radius $K+1$ centered at $z$ does not depend on
$x_{n}$.  If this statement were not true, then for some $x_{n}$ and
$x_{m}$, the geodesic segments from $y$ to $x_{n}$ and from $y$ to
$x_{m}$ would have to cross after their nearest approaches to $z$, by
planarity of $G^{\Gamma}$. This would contradict the geodesic property
for at least one of them.
\end{proof}

Now consider the terms $G_{r} (1,y)G_{r} (y,x)G_{r} (1,x)$ in the sum
\eqref{eq:sumPartition}. By the Ancona inequalities, the ratios 
\[
	\frac{G_{r} (1,z^{-} (y))G_{r} (z^{-} (y),y)}{G_{r} (1,y)}, \quad 
	\frac{G_{r} (y,z^{+} (y))G_{r} (z^{+},x)}{G_{r} (y,x)}, \quad
	\text{and} \quad 
	\frac{G_{r} (1,z)G_{r} (z,x)}{G_{r} (1,x)}
\]
are bounded away from $0$ and $\infty$, and by
Corollary~\ref{corollary:bi} and
Lemmas~\ref{lemma:geometry}--\ref{lemma:configuration} they depend
continuously on the local configuration $e^{-} (z),e^{+} (z)$ of the
geodesic $L (1,x)$ near $z$. By
Theorem~\ref{theorem:holderMartinKernel}, the ratios
\[
	\frac{G_{r} (1,z^{-} (y))}{G_{r} (1,z)},\;
	\frac{G_{r} (z^{-} (y),y)}{G_{r} (z,y)},\;
	\frac{G_{r} (y,z^{+} (y))}{G_{r} (z,y)},\;
	\text{and}
	\frac{G_{r} (z^{+},x)}{G_{r} (z,x)}
\]
also vary continuously with $e^{-} (z),e^{+} (z)$. Consequently, for a
suitable constant $\xi (r)$, the convergence \eqref{eq:ergodic}
follows from Corollary~\ref{corollary:ergodic}. That $\xi (r)$ varies
continuously with $r$ for $r\leq R$ follows from the continuous
dependence of the Gibbs state $\mu_{r}$ with $r$
(Proposition~\ref{corollary:ergodic}). It remains only to show that
$\xi (R) >0$; this follows from the  next lemma.

\begin{lemma}\label{lemma:density}
There exist  $K<\infty$ and $C>0$ independent of $1\leq r\leq R$ so that the
following is true. For any geodesic segment $L$ of length $\geq K$
corresponding to a path in the  Cannon automaton, and
any $K$ consecutive vertices $z_{1},z_{2},\dotsc , z_{K}$ on $L$,
\begin{equation}\label{eq:density}
	\sum_{j=1}^{K} \sum_{y\in \Gamma (z_{j})} G_{r} (z,y)^{2}
	\geq C \eta (r).
\end{equation}
\end{lemma}

\begin{proof}
This is in essence a consequence of hyperbolicity, but is easiest to
prove using symbolic dynamics.  Recall that the sphere $S_{m}$ of
radius $m$ in $\Gamma$ has the description \eqref{eq:sphereDecomp} by
words of length (approximately) $m$ in the Cannon automaton
$\mathcal{A}$.  Since the shift $(\Sigma ,\sigma)$ is topologically
ergodic and has positive topological entropy, there exists $K$ so
large that any path of length $\geq K$ has a \emph{fork} in the set of
recurrent vertices of $\mathcal{A}$, that is, a point where the path
could be continued in an alternative fashion. Let $\gamma$ be the path
in the automaton corresponding to $L$, and let $\gamma '$ be a path
(possibly much longer than $\gamma$) that agrees with $\gamma$ up to a
fork, where it then deviates from $\gamma$. If $K$ is sufficiently
large, then the geodesic $L'$ corresponding to $\gamma '$ will be such
that for every vertex $y\in L'$ the nearest vertex to $y$ in $L$ will
be one of the $K$ vertices $z_{1},\dotsc ,z_{K}$, by
Lemma~\ref{lemma:geometry}. Denote by $\beta'$ the segment  of $\gamma'$
following the fork from $\gamma$. Because the shift $(\Sigma ,\sigma)$
is topologically mixing, the set
of possible continuations $\beta '$ of length $m$ nearly coincides
with the set of paths $\beta'' $ such that for some short path $\alpha$
in $\mathcal{A}$ starting at $s_{*}$ the concatenation $\alpha \beta''$
is a path in $\mathcal{A}$. Thus, the sum in \eqref{eq:density}, which
(roughly) corresponds to the sum over all $\beta '$, is comparable to
the sum over all paths $\alpha \beta ''$ in $\mathcal{A}$.
\end{proof}

\bibliographystyle{plain}
\bibliography{mainbib}

%
%

\end{document}